\documentclass[11pt]{article} 
\usepackage{bbm,a4wide,subfigure,xcolor,srcltx}

\usepackage{amsmath,amssymb,amsthm,bbm,srcltx,a4wide,enumerate,xcolor,graphicx,subfigure}
\usepackage[colorlinks]{hyperref}

\newcommand{\Rset}{{\mathbb{R}}}
\newcommand{\pr}{{\mathbb P}}
\newcommand{\esp}{{\mathbb E}}

\newcommand\sphere[1]{\mathcal{S}^{#1}}

\newcommand\convweak {\ \stackrel{{\rm (d)}}\longrightarrow \ }

\newcommand\mbg{\mathbb{G}}
\newcommand\mbu{\mathbb{U}}
\newcommand\mbw{\mathbb{W}}
\newcommand\mbx{\mathbb{X}}
\newcommand\mby{\mathbb{Y}}

\newcommand\mct{\mathcal{T}}

\newcommand\rmd{\mathrm{d}}
\newcommand\rme{\mathrm{e}}
\newcommand\sgn{\mathrm{sgn}}
\newcommand\1[1]{\mathbbm{1}_{#1}}

\newcommand\diam[1]{M_n^{(#1)}}
\newcommand\diamq{M_{n,q}^{(2)}}
\newcommand\maxq{M_{n,q}}

\newcommand\beps{\boldsymbol{\varepsilon}}

\newcommand\psiT{\psi_T}
\newcommand\psiA{\psi_A}
\newcommand\phiA{\phi_A}
\newcommand\unifk{\mathbb{W}^{(k)}}

\newcommand{\goodgap}{\hspace{\subfigcapskip}}

\usepackage[utf8]{inputenc}
\usepackage{xspace}

\definecolor{Rcolor}{RGB}{150,160,190}

\newcommand{\Rx}{\fontsize{10pt}{12pt}\selectfont
\raisebox{.3em}{\hspace{1.2em}%
\llap{\resizebox{1.09em}{.5em}{\color{black}$\bigcirc$}}%
\llap{\resizebox{1.199em}{.55em}{\color{darkgray}$\bigcirc$}}%
\llap{\resizebox{1.19em}{.52em}{\color{gray!50}$\bigcirc$}}%
\llap{\resizebox{1.1em}{.5em}{\color{gray}$\bigcirc$}}%
\llap{\resizebox{1.25em}{.55em}{\color{gray}$\bigcirc$}}%
}%
\hspace{-.85em}%
\textbf{%
\textcolor{black}{\textsf{R}}%
\hspace{-.025em}\raisebox{.01em}{\llap{\textcolor{Rcolor}{\textsf{R}}}}%
}}%

\newbox\rbox
\savebox\rbox{\scalebox{0.1}{\Rx}}

\makeatletter
\def\Rlogo{\scalebox{\f@size}{\usebox\rbox}\xspace}
\makeatletter

\usepackage{cleveref}   

\newtheorem{theorem}{Theorem}[section]

\newtheorem{assumption}{Assumption}[section]
\newtheorem{lemma}[theorem]{Lemma}
\newtheorem{proposition}[theorem]{Proposition}
\newtheorem{corollary}[theorem]{Corollary}
\theoremstyle{remark}

\crefname{lemma}{lemma}{lemmas}
\Crefname{lemma}{Lemma}{Lemmas}
\crefname{theorem}{theorem}{theorems}
\Crefname{theorem}{Theorem}{Theorems}

\begin{document}

\title{The diameter of a random elliptical cloud}
\author{Yann Demichel* \and Ana Karina Ferm\'in* \and Philippe
  Soulier\thanks{Laboratoire Modal'X, Universit\'e Paris Ouest Nanterre, 200 avenue de la R\'epublique, 92000 Nanterre, France} }

\date{}
\maketitle

\begin{abstract}
  
  We study the asymptotic behavior of the diameter or maximum interpoint distance of a cloud of
  i.i.d.~$d$-dimensional random vectors when the number of points in the cloud tends to
  infinity. This is a non standard extreme value problem since the diameter is a max $U$-statistic,
  hence the maximum of  dependent random variables. Therefore, the limiting distributions may not be
  extreme value distributions. We obtain exhaustive results for the Euclidean diameter of a cloud of
  elliptical vectors whose Euclidean norm is in the domain of attraction for the maximum of the
  Gumbel distribution. We also obtain results in other norms for spherical vectors and we give
  several bi-dimensional generalizations. The main idea behind our results and their proofs is a
  specific property of random vectors whose norm is in the domain of attraction of the Gumbel
  distribution: the localization into subspaces of low dimension of vectors with a large norm. 
\end{abstract}


AMS Classification (2010): 60D05 60F05

\section{Introduction}

Let $\{\mbx,\mbx_i,i\geq1\}$ be i.i.d.~random vectors in $\Rset^d$, for a fixed $d\geq1$. The
quantities of interest in this paper are the maximum Euclidean norm $M_n(\mbx)$ and the Euclidean
diameter $\diam2(\mbx)$ of the sample, that is
\begin{align}
  M_n(\mbx) & = \max_{1 \leq i \leq n} \|\mbx_i\| \; , \\
  \diam2(\mbx) & = \max_{1 \leq i < j \leq n} \|\mbx_i - \mbx_j\| \; ,
\end{align}
where $\|\cdot\|$ denotes the Euclidean norm in $\Rset^d$.  The behavior of $M_n(\mbx)$ as $n$ tends
to infinity is a classical univariate extreme value problem. Its solution is well known. If the
distribution of $\|\mbx\|$ is in the domain of attraction of some extreme value distribution, then
$M_n(\mbx)$, suitably renormalized, converges weakly to this distribution.  We are interested in
this paper only in the case where the limiting distribution is the Gumbel law. More precisely, the
working assumption of this paper will be that there exist two sequences $\{a_n\}$ and $\{b_n\}$ such
that $\lim_{n\to\infty} a_n=\infty$, $\lim_{n\to\infty} b_n/a_n=0$ and
\begin{align}
   \label{eq:domaine}
  \lim_{n\to\infty}  n\pr(\|\mbx\|>a_n + b_n z) = \rme^{-z}
\end{align}
for all $z\in\Rset$, or  equivalently,
\begin{align*}
  \lim_{n\to\infty} \pr \left( \frac{M_n(\mbx) - a_n}{b_n} \leq z \right) =
  \rme^{-\rme^{-z}} \; .
\end{align*}
The asymptotic behavior of the diameter of the sample cloud is also an extreme
value problem since $\diam2(\mbx)$ is a maximum, but it is a non standard one, because
of the dependency between the pairs $(\mbx_i,\mbx_j)$.

This problem has been recently investigated by \cite{jammalamadaka:janson:2012}
for spherically distributed vectors, that is, vectors having the representation
$\mbx = T \mbw$ where $\mbw$ is uniformly distributed on the Euclidean unit
sphere $\sphere{d-1}$ of $\Rset^d$ and $T$ is a positive random variable in the
domain of attraction of the Gumbel distribution, independent of $\mbw$.  This
reference also contains a review of the literature concerning other domains of
attractions.

If $d=1$, a spherical random variable is simply a symmetric random variable, that is a positive
random variable multiplied by an independent random sign. The diameter of a real valued sample is
simply its maximum minus its minimum, and by independence and symmetry, it is straightforward to
check that $(\diam2(\mbx)-2a_n)/b_n$ converges weakly to the sum of two independent Gumbel random
variables with location parameter~$\log2$, i.e.~distributed as $\Gamma-\log2$, where $\Gamma$ is a
standard Gumbel random variable. Note that the tail of such a sum is heavier than the tail of one
Gumbel random variable.

If $d\geq2$, \cite{jammalamadaka:janson:2012} have shown that in spite of the dependency, the
limiting distribution is the Gumbel law, but a correction is needed. Precisely, they proved that
if~(\ref{eq:domaine}) holds, with an additional mild uniformity condition, there exists a sequence
$\{d_n\}$ such that $d_n\to\infty$, $d_n = O(\log(a_n/b_n))$ and
\begin{align}
  \lim_{n\to\infty} \pr \left( \frac{\diam2(\mbx) - 2a_n}{b_n} + d_n \leq z \right) =
  \rme^{-\rme^{-z}} \; . \label{eq:jj-spherique}
\end{align}
The exact expression of the sequence $\{d_n\}$ will be given in the comments after
Theorem~\ref{theo:diam-k>1}.  This implies that $\diam2(\mbx)/(2M_n(\mbx))$ converges in probability
to 1, but the behaviors of $M_n(\mbx)$ and $\diam2(\mbx)$ are subtly different. Specifically, $a_n$
is typically a power of $\log n$, so $\log(a_n/b_n)$ is of order $\log\log n$.

It is possible to give some rationale for the presence of the diverging correcting factor~$d_n$
in~(\ref{eq:jj-spherique}). In dimension one, two vectors with a large norm may be either on the
same side of the origin or on opposite sides. In the latter case their distance is automatically
large, typically twice as large as the norm of each one. In higher dimensions, two spherical vectors
with a large norm can be close to each other and their distance will be typically much smaller than
twice the norm of the largest one. Therefore we expect the probability that the diameter is large to
be smaller in the latter case.

This suggests that the asymptotic behavior of the diameter is related to the localization of vectors
with large norm. The behavior will differ if large values are to be found in some specific regions
of the space or can be found anywhere.

There are many possible directions to extend the results of
\cite{jammalamadaka:janson:2012}.  One very simple case not covered by these
results is the multivariate Gaussian distribution with correlated components.
The Gaussian distribution is a particular case of elliptical distributions.  The
main purpose of this paper is to investigate the behavior of the diameter of a
sample cloud of elliptical vectors.

Elliptical vectors are widely used in extreme value theory since they are in the
domain of attraction of multivariate extreme value distributions. These
distributions and their generalizations have been recently considered in the
apparently unrelated problem of obtaining limiting conditional distributions
given one component is extreme, see \cite{MR2739357} and the references therein.

In this paper, the tail behavior of a product $TU$, where $T$ is in the domain of
attraction of the Gumbel distribution and $U$ is a bounded positive random
variable independent of~$T$, was obtained as a by-product of the main
result. Under some regularity assumption on the density of $U$ at its maximum,
the tail of $TU$ is slightly lighter than the tail of $T$. The main 
reason is that if a random variable $T$ is in the domain of attraction of the Gumbel
distribution, then  for any $\alpha>1$,
\begin{align*}
  \lim_{x\to\infty} \frac{\pr(T>\alpha x)}{\pr(T>x)}  = 0  \; .
\end{align*}
This implies that for $TU$ to be large, $U$ must be very close to its maximum.  The full strength of
this remark was recently exploited in \cite{barbe:seifert:2013} who obtained the rate of convergence
of $U$ towards its maximum when the product $TU$ is large and the conditional limiting distribution
of the difference between $U$ and its maximum, suitably renormalized. This property explains deeply
the conditional limits obtained in \cite{MR2739357}.  Having in mind the earlier remarks on the link
between the localization of the vectors with large norm and the asymptotic distribution of the
diameter, it is clear that this localization property will be helpful to study the problem at hand
in this paper.

The rest of the paper is organized as follows. In Section~\ref{sec:elliptic}, we will define
elliptical vectors and state our main results. In section~\ref{theo:norm-ddim}, extending the
results of~\cite{barbe:seifert:2013}, we will show that the realizations of a $d$-dimensional random
elliptical vector with large norm are localized on a subspace of~$\Rset^d$ whose dimension is the
multiplicity of the largest eigenvalue of the covariance matrix. This result will be crucial to
prove our main results which are stated in Section~\ref{sec:diameter}. As partially conjectured by
\cite[Section~5.4]{jammalamadaka:janson:2012}, if the largest eigenvalue of the covariance matrix is
simple, then the limiting distribution of the diameter is similar but not equal to the one which
arises when~$d=1$: correcting terms appear that are due the fluctuations around the direction of the
largest eigenvalue. If the largest eigenvalue is not simple, say its multiplicity is $k$, then the
diameter behaves as in the spherical case in dimension $k$, up to constants.

In Section~\ref{sec:lq-spherical}, we will answer another question of
\cite{jammalamadaka:janson:2012}, namely we will investigate the $l^q$ diameter of a cloud of
spherical vectors, for $1 \leq q \leq \infty$. This problem is actually simpler than the
corresponding one in Euclidean ($l^2$) norm, since the vectors with large norm are always localized
close to a finite number of directions. Therefore, the ``localization principle'' applies and we
obtain the same type of limiting distribution as in the case of an elliptical distribution with
simple largest eigenvalue. For $q=1$ and $q>2$ the problem simplifies even further since the
corrective terms vanish and the limiting distribution of the one dimensional case is obtained.

In Section~\ref{sec:bidim}, we discuss further possible generalizations and give several
bidimensional examples.

We think that beyond answering certain questions on the diameter of a random cloud, the main purpose
of this paper is to emphasize the use of the localization principle of vectors with large norm in
the domain of attraction of the Gumbel distribution. This principle should be useful in other problems.

\section{The Euclidean norm of an elliptical vector}
\label{sec:elliptic}
A random vector $\mbx$ in $\Rset^d$ has an elliptical distribution if it can be expressed as
\begin{align}
  \label{eq:def-elliptic}
  \mbx=TA\mbw
\end{align}
where $T$ is a positive random variable, $A$ is an invertible $d\times d$ matrix and $\mbw=(W_1,
\dots,W_d)$ is uniformly distributed on the sphere $\sphere{d-1}$. The covariance matrix of $\mbx$
is then given by $\esp[T^2] A'A$ where $M'$ denotes the transpose of any matrix $M$. Let $\lambda_1
\geq \dots \geq \lambda_d > 0$ be its ordered eigenvalues repeated according to their
multiplicity. The distribution is spherical if all the eigenvalues are equal. Otherwise, there
exists $k\in\{1,\dots,d-1\}$ such that
\begin{align}
  \label{eq:eignvalues}
  \lambda_1 = \cdots = \lambda_k > \lambda_{k+1} \; .
\end{align}
We will see that this number $k$ plays a crucial role for tail of the norm and the asymptotic
distribution of the diameter.

Let $\mbw_{i}=(W_{i,1}, \dots,W_{i,d})$, $i=1,2$, be independent random vectors uniformly
distributed on $\sphere{d-1}$, and define $\mbx_i=T_i A\mbw_i$, which are i.i.d.~with the same
distribution as $\mbx$.  Since for any orthogonal matrix $P$ (i.e. $P'=P^{-1}$), $P\mbw$ is also
uniformly distributed on $\sphere{d-1}$, it holds that
\begin{align*}
  \|\mbx\|^2 & \stackrel{(d)} = T^2 \sum_{q=1}^d \lambda_q W_q^2 \; ,  \\
  \|\mbx_1-\mbx_2\|^2 & \stackrel{(d)}= T^2 \sum_{q=1}^d \lambda_q (W_{1,q}-W_{2,q})^2 \; ,
\end{align*}
where $\stackrel{(d)}=$ denotes equality in law. Define
\begin{align}
  \label{eq:def-y}
  \mby=T(\sqrt{\lambda_1}W_1,\dots,\sqrt{\lambda_d}W_d)
\end{align}
and let $\{\mby_i,i\geq1\}$ is a sequence of i.i.d.~vectors with the same distribution as $\mby$.
Then
\begin{align*}
  M_n(\mbx) \stackrel{(d)}= M_n(\mby) \; , \ \ \diam2(\mbx) \stackrel{(d)}=\diam2(\mby) \;.
\end{align*}
Therefore, we will prove our results using the vectors $\{\mby_i,i\geq1\}$.

In all the sequel, we will assume that $T$ is in the max domain of attraction of the Gumbel law,
i.e.~the limit~(\ref{eq:domaine}) holds, or equivalently, there exists a function $\psiT$, called an
auxiliary function for $T$, defined on $(0,\infty)$ such that
\begin{align*}
  \lim_{x\to\infty} \frac{\psiT(x)}x = 0 \; ,
\end{align*}
and
\begin{align}
   \label{eq:gumbel-DA-psi}
   \lim_{x\to\infty} \frac{\pr(T>x+\psiT(x)z)}{\pr(T>x)} = \rme^{-z} \; ,
\end{align}
locally uniformly with respect to $z\in\Rset$.  Moreover, the survival function
of $T$ can be expressed as
\begin{align}
  \label{eq:vonmises}
  \pr(T>x) = \vartheta(x) \exp \left(-\int_{x_0}^x \frac{\rmd s}{\psiT(s)} \right) \; , 
\end{align}
where $\lim_{x\to\infty} \vartheta(x) \in (0,\infty)$. See e.g.~\cite[Chapter~0]{resnick:1987}.

Define the functions $\psiA$ and $\phiA$ on $(0,\infty)$ by
\begin{align*}
  \psiA(x) & = \sqrt{\lambda_1} \psiT(x/\sqrt{\lambda_1}) \; , \ \ \ \phiA(x) =
  \sqrt{\frac{\psiT(x/\sqrt{\lambda_1})} {x/\sqrt{\lambda_1}}} = \sqrt{\psiA(x)/x} \; .
\end{align*}

In the sequel, the notation $\sim $ means that the ratio of the two terms around $\sim$ tend to one
when their parameter ($x$ or $n$) tends to infinity and $\convweak$ denotes weak convergence of
probability distribution.

\begin{theorem}
  \label{theo:norm-ddim}
  Let $\mbx$ be as in~(\ref{eq:def-elliptic}) with $T$ satisfying~(\ref{eq:gumbel-DA-psi}), $\mbw$
  uniformly distributed on~$\sphere{d-1}$, and assume that the eigenvalues
  $\lambda_1,\dots,\lambda_d$ of the correlation matrix satisfy~(\ref{eq:eignvalues}). Then,
\begin{align*}
  \pr(\|\mbx\|>x) & \sim \frac{\Gamma(\frac d2)}{ \Gamma(\frac{k}2)} 2^{(d-k)/2} \;
  \left(\prod_{q=k+1}^d \frac{\lambda_1}{\lambda_1-\lambda_q}\right)^{1/2} \; \phi_A^{d-k}(x) \;
  \pr(T>x/\sqrt{\lambda_1}) \; .
\end{align*}
Let $\mby$ be as in~(\ref{eq:def-y}). Then, as $x\to\infty$, conditionally on $\|\mby\|>x$,
\begin{align*}
  \left(
    \tfrac{\|\mby\|-x}{\psi_A(x)},W_1,\dots,W_k,\tfrac{W_{k+1}}{\phi_A(x)},\dots,\tfrac{W_d}{\phi_A(x)}
  \right) \convweak (E,\unifk,\sqrt{\tfrac{\lambda_1}{\lambda_1-\lambda_{k+1}}}G_{k+1}, \dots,
  \sqrt{\tfrac{\lambda_1}{\lambda_1-\lambda_{d}}}G_d) \; ,
\end{align*}
where $E$ is an exponential random variable with mean~1, $\unifk$ is
uniformly distributed on $\sphere{k-1}$, $G_{k+1},\dots,G_d$ are
independent standard Gaussian random variables, and all components are
independent.
\end{theorem}

\paragraph{Comments}
\begin{itemize}
\item This result implies that $\|\mbx\|$ is in the domain of attraction of the Gumbel distribution
  and that an auxiliary function for $\|\mbx\|$ is $\psiA$.

\item The first statement can be obtained as a consequence of
  \cite[Proposition~3.2.1]{MR2739357}. In dimension 2, the second result is a consequence of
  \cite[Theorem~2.1]{barbe:seifert:2013}, where a real valued random variable $X$ which can be
  expressed $X = Tu(S)$ is considered, with $T$ satisfying~(\ref{eq:vonmises}), $S$ taking values in
  $[0,1]$ and the bounded function~$u$ having some regularity properties around its maximum and the
  asymptotic behavior of $S$ conditionally on the product $Tu(S)$ being large is obtained.
\end{itemize}

We now consider the polar representation of the vector $\mby$, that is we define $\Theta =
\frac{\mby}{\|\mby\|}$ and for $q=k+1,\dots,d$, we define also $ \tau_q^2 =
{\frac{\lambda_q}{\lambda_1-\lambda_q}}$.

\begin{corollary}
  \label{coro:norm-ddim} Under the conditions of Theorem~\ref{theo:norm-ddim}, as $x\to\infty$,
  conditionally \mbox{on $\|\mby\|>x$},
\begin{align*}
  \left( \tfrac{\|\mby\|-x}{\psiA(x)},\Theta_1, \dots, \Theta_k,\tfrac{\Theta_{k+1}}
    {\phiA(x)},\dots,\tfrac{\Theta_d}{\phi_A(x)}\right) \convweak
  (E,\unifk,\tau_{k+1}G_{k+1},\dots,\tau_dG_d) \; ,
\end{align*}
where $E$ is an exponential random variable with mean~1, $\unifk$ is uniformly distributed on
$\sphere{k-1}$, $G_{k+1},\dots,G_d$ are independent standard Gaussian random variables, and all
components are independent.
\end{corollary}

This result can be rephrased in terms of weak convergence of point processes (see
e.g.~\cite[Proposition~3.21]{resnick:1987}).  Let $a_n$ be the $1-1/n$ quantile of the distribution
of $\|\mbx\|$ or $\|\mby\|$, i.e.  $\pr(\|\mbx\|>a_n) = \pr(\|\mby\|>a_n) \sim 1/n$ and set $b_n =
\psiA(a_n)$ and $c_n = \phiA(a_n)$.  Define the points
\begin{align}
  \label{eq:def-Pni}
  P_{n,i} = \left( \frac{\|\mby_i\|-a_n}{b_n}, \Theta_1, \dots,\Theta_k, \frac{\Theta_{k+1}}{c_n},
    \dots, \frac{\Theta_d}{c_n} \right) \; .
\end{align}

\begin{corollary}
  \label{coro:ppconv}
  Under the conditions of Theorem~\ref{theo:norm-ddim}, the point processes $\sum_{i=1}^n
  \delta_{P_{n,i}}$ converge weakly to a Poisson point process $N = \sum_{i=1}^\infty\delta_{P_i}$
  on $\Rset\times\sphere{k-1}\times\Rset^{d-k}$ with
  \begin{align}
    \label{eq:def-Pi}
    P_i = (\Gamma_i,\unifk_{i},\tau_{k+1} G_{i,k+1},\dots,\tau_dG_{i,d}) \; ,
  \end{align}
  where $\{\Gamma_{i},i\geq1\}$ are the points of a Poisson point process on $(-\infty,\infty]$ with
  mean measure $\rme^{-x} \rmd x$, $\{\unifk_{i},i\geq1\}$ are i.i.d.~vectors uniformly distributed
  on $\sphere{k-1}$ and $\{G_{i,q}, i\geq1, q=k+1,\dots,d\}$ are i.i.d.~standard Gaussian variables,
  all sequences being mutually independent.

\end{corollary}

\paragraph{Comments} Since the measure $\rme^{-x}\,\rmd x$ is finite on any interval
  $[a,\infty]$, $a\in\Rset$, the point process $N$ has a finite number of points on any set
  $[a,\infty] \times \sphere{d-1} \times \Rset^{d-k}$. Therefore, the points can and will be
  numbered in such a way that $\Gamma_1>\Gamma_2>\dots$. Moreover, if the points $P_{n,i}$ are also
  numbered in decreasing order of their first component, then for each fixed integer $m$,
  $(P_{n,1},\dots,P_{n,m})$ converges weakly to $(P_1,\dots,P_m)$.

We illustrate Theorem~\ref{theo:norm-ddim} for three dimensional Gaussian vectors whose maximum
eigenvalue $\lambda_1$ of the correlation matrix is simple (Figure~\ref{fig:oursin}a) or double
(Figure~\ref{fig:oursin}b). The rate of convergence to zero of the coordinates corresponding to the
smallest eigenvalues is $O(\log n)$.

\clearpage 

\begin{figure}[!t]
\centering 
\subfigure[$\lambda_1=4$, $\lambda_2=1$, $\lambda_3=0.5$]
{\includegraphics[scale=0.30,clip=true,trim=80 80 50 80]{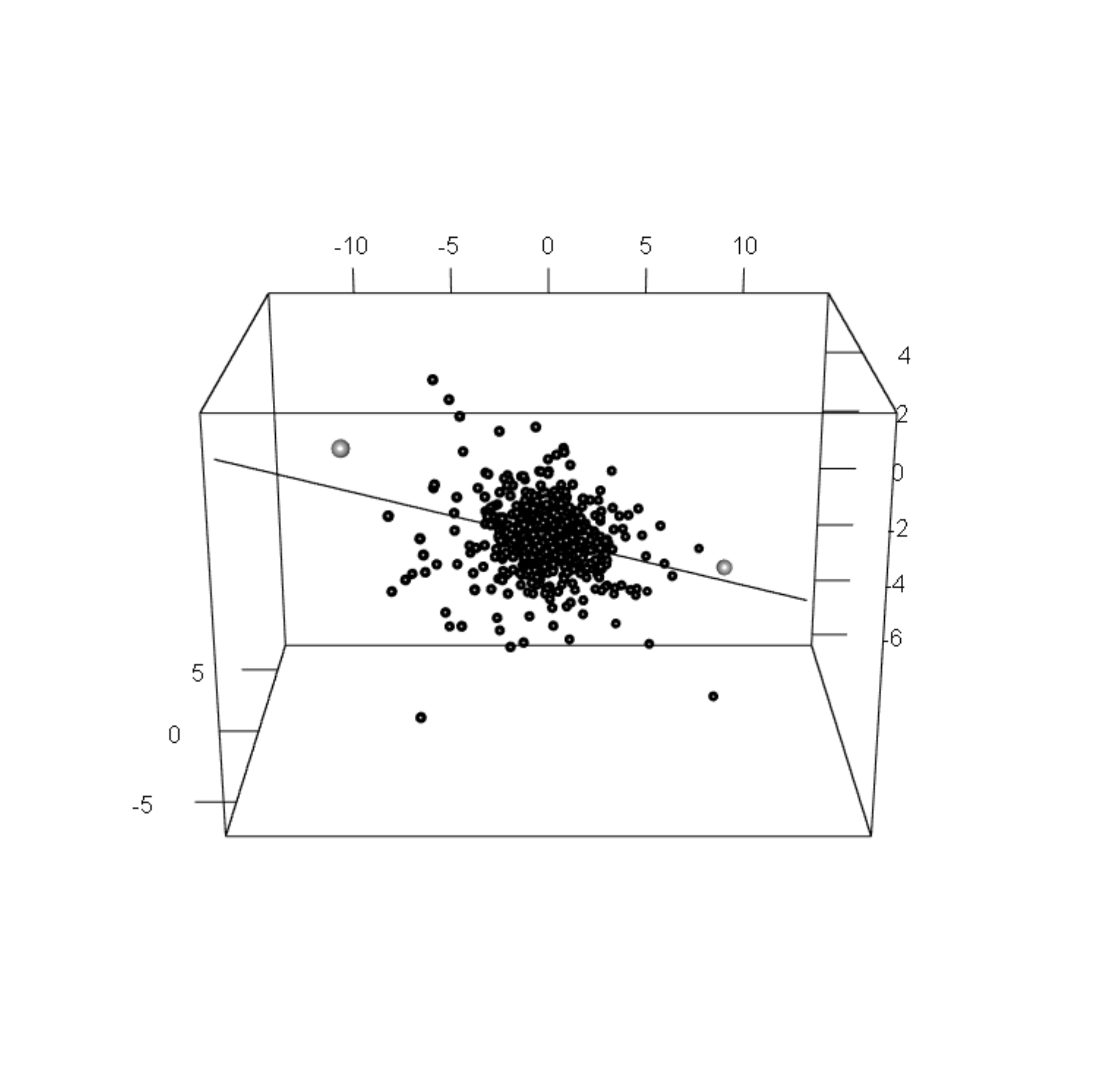}} 
\goodgap
\subfigure[$\lambda_1=\lambda_2=4$, $\lambda_3=0.5$]
{\includegraphics[scale=0.30,trim=80 80 80 80]{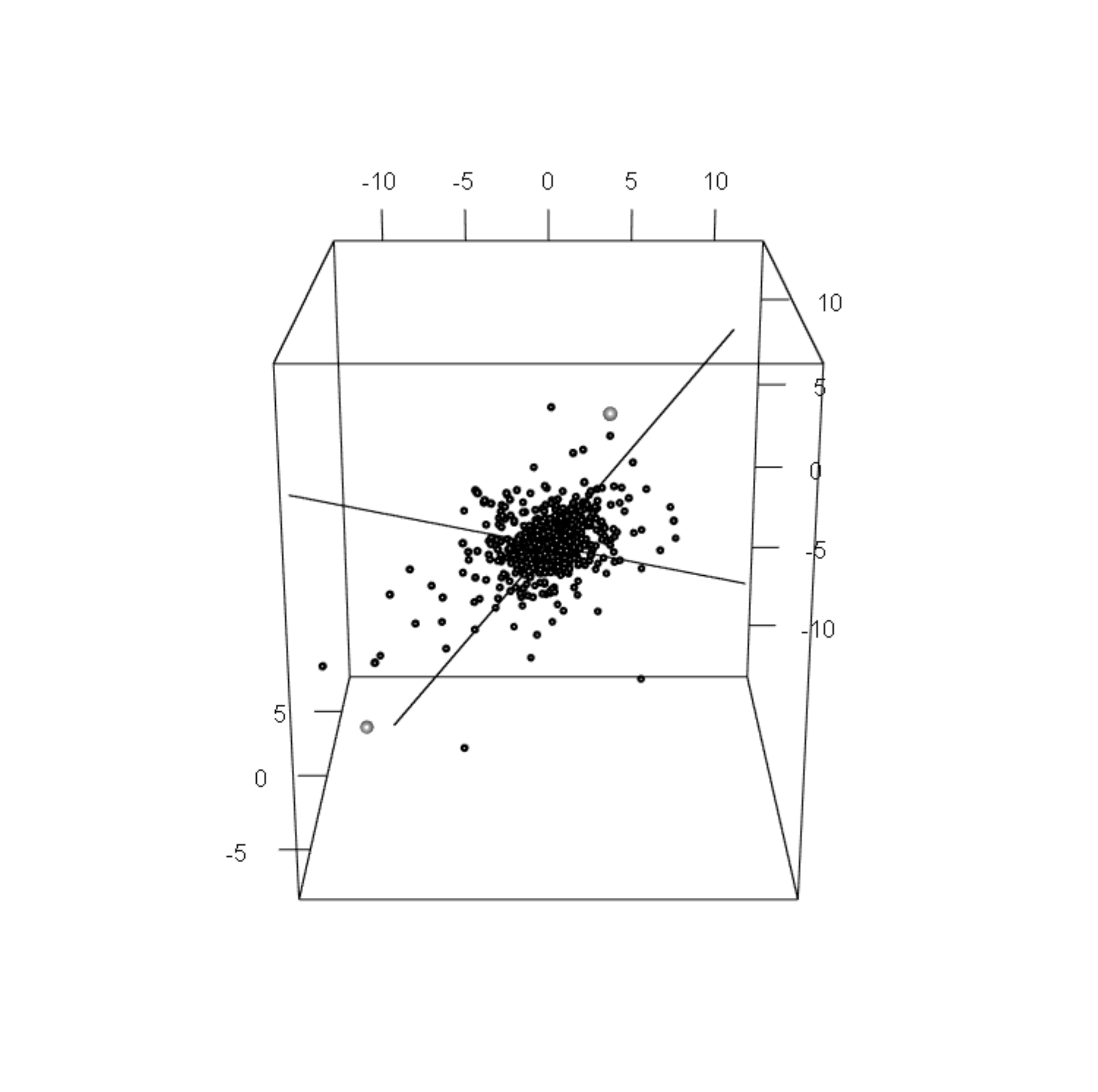}}
\vspace{-0.2cm}
\caption{\footnotesize{Two sample clouds of size $1000$ of the trivariate Gaussian distribution. The
    eigenvalues of the correlation matrix are given under each figure. The gray spheres are the
    points which realize the diameter.  The black lines are the principal axes.}} \label{fig:oursin}
\end{figure}

\subsubsection*{Proof of Theorem~\ref{theo:norm-ddim}}

We will need the following Lemma.

\begin{lemma}
  \label{lem:conditioning-uniform-sphere}
  Let $\mbw$ be uniformly distributed on $\sphere{d-1}$. For $k\in\{1,\dots,d-1\}$, define the
  random vector $\unifk$ on $\sphere{k}$ by
  \begin{align*}
    \unifk = \frac1{\sqrt{1-W_{k+1}^2-\cdots-W_d^2}} (W_{1},\dots,W_k) \; .
  \end{align*}
  Then $\unifk$ is uniformly distributed on $\sphere{k}$ and independent of $(W_{k+1},\dots,W_d)$.
  If $f$ is continuous and compactly supported on $\Rset^d$, then
  \begin{multline}
    \label{eq:spherical-dilate}
    \lim_{s\to\infty} s^{d-k} \, \esp [ f(W_1,\dots,W_k,sW_{k+1},\dots,sW_d) ] \\
    = \frac{\Gamma(\frac d2)}{\pi^{\frac{d-k}2} \Gamma(\frac{k}2)} \int_{\Rset^k}
    \esp[f(\unifk,t_{k+1},\dots,t_d)] \, \rmd t_{k+1} \dots \rmd t_d \; .
  \end{multline}
\end{lemma}

The convergence~(\ref{eq:spherical-dilate}) can be extended to sequences of continuous functions
$f_x$ with compact support which depend on $x$ provided they converge locally uniformly to a
continuous function with compact support. By bounded convergence, it can also be extended to
sequences of bounded continuous functions $f_x$ if there exists a function $f^*$ (not depending on
$x$) and such that $|f_x|\leq f^*$ for all $x$ and $\int_{\Rset^k}
\esp[f^*(\unifk,t_{k+1},\dots,t_d)] \, \rmd t_{k+1} \dots \rmd t_d < \infty$.  The proof of the
Lemma consists merely in a change of variable and is postponed to Section~\ref{sec:proof-lemmas}.

\begin{proof}[Proof of Theorem~\ref{theo:norm-ddim}]
  Note first that if $(w_1,\dots,w_d) \in \sphere{d-1}$, then
  \begin{align*}
    \sum_{q=1}^d \lambda_q w_q^2 = \lambda_1 \left(1 - \sum_{q=k+1}^d \gamma_q^{-2} w_q^2 \right) \;    ,
  \end{align*}
  with $\gamma_q^2 = \lambda_1/(\lambda_1-\lambda_q)$, $q=k+1,\dots,d$. Thus we can write
  \begin{align*}
    \|\mby\| = \sqrt{\lambda_1}T - \sqrt{\lambda_1}T g(W_{k+1},\dots,W_d) \; ,
  \end{align*}
  where
  \begin{align*}
    g(w_{k+1},\dots,w_d) & = 1-\sqrt{1 - \sum_{q=k+1}^d \gamma_q^{-2} w_q^2} \; ,
  \end{align*}
  and
  \begin{align*}
    \lim_{s\to\infty} s^2 g(s^{-1}w_{k+1},\dots,s^{-1} w_d) & = \frac12 \sum_{q=k+1}^d \gamma_q^{-2}
    w_q^2 \; ,
  \end{align*}
  locally uniformly with respect to $w_{k+1},\dots,w_d$.

  For $x>0$, define the function $k_x$ on $\Rset^{d-k+1}$ by
  \begin{align*}
    k_x(z,w_{k+1},\dots,w_d) & = \frac{\pr \left( \|\mbx\|> x + \psiA(x)z \mid
        W_q=\phiA(x)w_q,q=k+1,\dots,d \right)}{\pr(T>x/\sqrt{\lambda_1})}  \\
    & = \frac{\pr \left( T > \frac{x/\sqrt{\lambda_1} \; + \;\psiT(x/\sqrt{\lambda_1})z} { 1
          -g(\phiA(x)w_{k+1},\dots,\phiA(x)w_d)} \right)} {\pr(T>x/\sqrt{\lambda_1})} \; .
  \end{align*}
  Since we have defined $\phiA$ such that $x\phiA^2(x) = \sqrt{\lambda_1}
  \psiT(x/\sqrt{\lambda_1})$, we obtain that
  \begin{align*}
    \lim_{x\to\infty} k_x(z,w_{k+1},\dots,w_d) = \exp\bigg(-z-\frac12\sum_{q=k+1}^d
    \gamma_q^{-2}w_q^2\bigg) \; ,
  \end{align*}
  locally uniformly with respect to $z,w_{k+1},\dots,w_d$.
  Let $f$ be a continuous function with compact support in $\Rset^d$. Applying
  Lemma~\ref{lem:conditioning-uniform-sphere}, we obtain
  \begin{align}
    \lim_{x\to\infty} & \frac1{\phiA^{k-d}(x)\pr(T>x/\lambda_1)} \esp \left[
      f\left(W_1,\dots,W_k,\tfrac{W_{k+1}}{\phi(x)}, \dots, \tfrac{W_{k+1}}{\phi(x)} \right)
      \1{\{\|\mby\|>x+\psiA(x)z\}} \right] \nonumber \\
    & = \lim_{x\to\infty} \phiA^{k-d}(x) \esp \left[ f\left(W_1,\dots,W_k,\tfrac{W_{k+1}}{\phi(x)},
        \dots, \tfrac{W_{k+1}}{\phi(x)} \right) k_x\left(z,\tfrac{W_{k+1}}{\phiA(x)}, \dots,
        \tfrac{W_d}{\phiA(x)}\right) \right] \nonumber \\
    & = \frac{\Gamma(\frac d2)}{\pi^{(d-k)/2} \Gamma(\frac{k}2)} \; \rme^{-z} \, \int_{\Rset^{d-k}}
    \esp[f(\unifk,t_{k+1},\dots,t_d)] \, \rme^{-\frac12 \sum_{q=k+1}^d
      \gamma_q^{-2}w_q^2} \,  \rmd t_{k+1} \dots \rmd t_d \nonumber \\
    & = \frac{\Gamma(d/2)}{\Gamma(k/2)} 2^{({d-k})/2} \; \prod_{q=k+1}^d \gamma_q \; \rme^{-z} \,
    \esp[f(\unifk,\gamma_{k+1}G_{k+1},\dots,\gamma_dG_d)] \; ,
    \label{eq:desired}
  \end{align}
  where $G_{k+1},\dots,G_d$ are i.i.d.~standard Gaussian random variables.

  The last step is to extend the convergence~(\ref{eq:desired}) to all bounded continuous functions
  $f$.  By the comments after Lemma~\ref{lem:conditioning-uniform-sphere}, it suffices to prove that
  the function $k_x$ can be bounded by a function $k^*$ independent of $x$ and integrable with
  respect to Lebesgue's measure on~$\Rset^{d-k}$.
  For any $u\geq0$ and $p>0$, there exists a constant $C$ such that, for large enough $x$,
  \begin{align}
    \label{eq:borne-uniforme}
    \frac{ \pr(T> x + \psiT(x)u)}{\pr(T>x)} \leq C (1+u)^{-p} \; ,
  \end{align}
  (see e.g.~\cite[Lemma~5.1]{MR2739357}).  For $z\geq0$, this trivially yields
  \begin{align}
    \label{eq:borne-uniforme-avec-z}
    \frac{ \pr(T> x + \psiT(x)(z+u))}{\pr(T>x)} \leq C (1+u)^{-p} \; .
  \end{align}
  For a fixed $z<0$, we write
  \begin{align*}
    \frac{\pr(T>x+\psiT(x)(z+u))}{\pr(T>x)} = \frac{\pr(T>x+\psiT(x)z)}{\pr(T>x)}
    \frac{\pr(T>x+\psiT(x)z + \psiT(x)u)}{\pr(T>x+\psiT(x)z)} \; .
  \end{align*}
  The first ratio in the right hand side is convergent hence bounded and since $z$ is fixed, we can
  apply the bound~(\ref{eq:borne-uniforme}) to the second ratio, upon noting that $\lim_{x\to\infty}
  \psiT(x+\psiT(x)z)/\psiT(x)=1$ for all $z\in\Rset$. Thus (\ref{eq:borne-uniforme-avec-z}) also
  holds with a constant $C$ uniform with respect to $z$ in compact sets of $(-\infty,0]$.

  Since
  \begin{align*}
    \frac1{1 - g(w_{k+1},\dots,w_d)} \geq 1 + \frac12 \sum_{q=k+1}^d \gamma_q^{-2} w_q^2 \; ,
  \end{align*}
  we obtain, applying~(\ref{eq:borne-uniforme-avec-z}) with $u= \frac12 \sum_{q=k+1}^d \gamma_q^{-2}
  w_q^2$ and a fixed $z\in\Rset$,
  \begin{align*}
    k_x(z,w_{k+1},\dots,w_d) \leq \frac{\pr \left( T > x/\sqrt{\lambda_1} +\psiT(x/\sqrt{\lambda_1})
        (z+u) \right) } {\pr(T>x/\sqrt{\lambda_1})} \leq C (1+u) ^{-p} \; .
  \end{align*}
  For $p$ large enough, the function $k^*(w_1,\dots,w_d) =\left(1+\frac12\sum_{q=k+1}^d
    \gamma_q^{-2}w_q^2\right)^{-p}$ is integrable with respect to Lebesgue's measure on
  $\Rset^{d-k}$. This concludes the proof.
\end{proof}

\section{Asymptotic behavior of the Euclidean diameter}
\label{sec:diameter}

We now study the behavior of the diameter of the elliptical cloud $\{\mbx_i,1\leq i\leq
n\}$. Precisely, we investigate the asymptotic behavior of $(\diam2(\mbx)-2a_n)/b_n$ in the case
$k=1$ and $k>1$. As previously, we will prove our results with the vectors $\mby_i$, $i\geq1$.

\subsection{Case $k=1$: single maximum eigenvalue}

In this case, the points $P_{n,i}$ defined in~(\ref{eq:def-Pni}) become
\begin{align*}
  P_{n,i} = \left( \frac{\|\mby_i\|-a_n}{b_n}, \Theta_1, \frac{\Theta_{2}}{c_n},\dots,
    \frac{\Theta_d}{c_n} \right) \; .
\end{align*}

By Corollary~\ref{coro:ppconv}, $N_n = \sum_{i=1}^n \delta_{P_{n,i}}$ converges weakly to a Poisson
point process $N = \sum_{i=1}^\infty \delta_{P_i}$ on $\Rset\times\{-1,1\}\times\Rset^{d-1}$ with
$P_i = (\Gamma_i,\varepsilon_i,\tau_2G_{i,2},\dots,\tau_d G_{i,d})$, where $\varepsilon_i$, $i\geq1$
are i.i.d.~symmetric random variables with values in $\{-1,+1\}$ and the other components are as in
Corollary~\ref{coro:ppconv}.

By the independent increment property of the Poisson point process, the point process $N$ can be
split into two independent Poisson point processes $N^+$ and $N^-$ on $\Rset^d$ whose points are the
points of $N$ with second component equal to $+1$ or $-1$ respectively. The mean measure of both
processes is $\frac12\rme^{-x}\,\rmd x \, \Phi_{\tau_{2}}(\rmd t_{2}) \cdots \Phi_{\tau_d}(\rmd
t_d)$.

Then the point processes $N_n^+$ and $N_n^-$ defined by 
\begin{align*}
  N_n^+ = \sum_{i=1}^n \delta_{P_{n,i}} \1{\{\Theta_{i,1}>0\}} \; , \ \ N_n^- = \sum_{i=1}^n
  \delta_{P_{n,i}} \1{\{\Theta_{i,1}<0\}} \; ,
\end{align*}
converge weakly to the independent point processes $N^+$ and $N^-$ on $\Rset^{d}$ which can be
expressed as
\begin{align*}
  N^\pm = \sum_{i=1}^\infty \delta_{(\Gamma_i ^\pm,\tau_2G_{i,2}^\pm,\dots,\tau_dG_{i,d}^\pm)} \; ,
\end{align*}
where $\{\Gamma_i^\pm, i\geq1\}$ are the points of a Poisson point process with mean measure
$\frac12\rme^{-x}\rmd x$ on $\Rset$, and $\{G_{i,q}^\pm$, $i\geq1$, $q=2,\dots,d\}$ are
i.i.d.~standard Gaussian variables, independent of the points $\{\Gamma_i^\pm, i\geq 1\}$.

Since the mean measure is finite on the half planes $(x,\infty]\times[-\infty,\infty]$, there is
almost surely a finite number of points of $N^\pm$ in any of these half planes. Thus, the points of
$N^\pm$ can and will be numbered in decreasing order of their first component.

We can now state the main result of this section.
\begin{theorem}
  \label{theo:diam-k=1}
  Let $\{\mbx_i,i\geq1\}$ be a sequence of i.i.d.~random vectors with the same distribution as
  $\mbx$ and let the assumptions of Theorem~\ref{theo:norm-ddim} hold with $k=1$, i.e.~$\lambda_1>\lambda_2$, then
  \begin{align}
    \label{eq:limit-diam-k=1}
    \frac{\diam2(\mbx)-2a_n}{b_n} \convweak \max_{i,j\geq1} \left\{ \Gamma_i^+ + \Gamma_j^- -
      \frac14\sum_{q=2}^d \frac{\lambda_q}{\lambda_1-\lambda_q} (G_{i,q}^+-G_{j,q}^-)^2 \right\} \;    ,
  \end{align}
  where $\{(\Gamma_i^+,G_{i,2}^+,\dots,G_{i,d}^+), i\geq 1\}$ and
  $\{(\Gamma_i^-,G_{i,2}^-,\dots,G_{i,d}^-), i\geq1\}$ are the points of two independent point
  processes with mean measure $\frac12\rme^{-x} \rmd x \Phi(\rmd t_2)\dots\Phi(\rmd t_d)$.
\end{theorem}

\paragraph{Comments}
The random variable defined in~(\ref{eq:limit-diam-k=1}) is almost surely finite, since it is upper
bounded by $\Gamma_1^++\Gamma_1^-$. The lower bound $\Gamma_1^++\Gamma_1^- - \frac14 \sum_{q=2}^d
\tfrac{\lambda_q}{\lambda_1-\lambda_q} (G_{1,q}^+-G_{1,q}^-)^2$ trivially holds. These two bounds
imply that the limiting distribution is tail equivalent to the sum of two independent Gumbel random
variables which is heavier tailed than a Gumbel distribution. However, it is not the sum of two
independent Gumbel random variables. Therefore this result is different from the result in the
spherical case in any dimension.

\subsubsection{Case of the dimension~2}

In dimension 2, a bivariate elliptical random vector $\mbx$ with correlation $\rho\in(0,1)$ can be defined by
\begin{align*}
  \mbx = T(\cos U, \rho\cos U + \sqrt{1-\rho^2} \sin U) = T(\cos U, \cos (U-U_0))\; ,
\end{align*}
where $U$ is uniformly distributed on $[0,2\pi]$, $\cos U_0 = \rho$ and $\sin U_0 =
\sqrt{1-\rho^2}$. The vector $\mbx$ admits the polar representation $\mbx=R(\cos\Theta,\sin\Theta)$
with
\begin{align*}
  R & = T \sqrt{1+\rho \cos(2U-U_0)} \; , \\
  \cos \Theta & = \frac{\cos U}{\sqrt{1+\rho \cos(2U-U_0)}} \; , \sin \Theta =
  \frac{\cos(U-U_0)}{\sqrt{1+\rho \cos(2U-U_0)}} \; .
\end{align*}

The correlation matrix of $\mbx$ is then
\begin{align*}
  \begin{pmatrix}1 & \rho \\ \rho & 1 \end{pmatrix} \; .
\end{align*}
Its eigenvalues are $1+\rho$ and $1-\rho$. By Theorem~\ref{theo:norm-ddim}, we know that $\|\mbx\|$
is in the domain of attraction of the Gumbel law and more precisely, as $x\to\infty$,
\begin{align*}
  \pr(\|\mbx\| >x\sqrt{1+\rho} ) \sim \sqrt{\frac{1+\rho}{\pi\rho} }\sqrt{\frac{\psiT(x)}{x}}
  \pr\left(T>x\right) \; .
\end{align*}

Note that $(1,1)$ is always an eigenvector associated with the eigenvalue $1+\rho$. This means that
the vectors in the cloud with large norm are localized close to the diagonal, whatever the value of
$\rho\in(0,1)$. More precisely, let $\widetilde{\Theta}_n$ be the angle of the point
$\widetilde{\mbx}_n$ of the cloud $\{\mbx_i,1\leq i\leq n\}$ such that
$\|\widetilde{\mbx}_n\|=M_n(\mbx)$.  Then $(\widetilde\Theta_n-\tfrac\pi4 -
\pi\1{\{\cos\widetilde\Theta_n<0\}})/c_n$ converges weakly to a Gaussian variable with mean zero and
variance $(1-\rho)/2\rho$.

By Theorem~\ref{theo:diam-k=1}, the limiting distribution of the diameter can be expressed as
\begin{align}
  \label{eq:loi-diam-d2}
  \max_{i,j\geq1} \left\{ \Gamma_i^+ + \Gamma_j^- - \frac{1-\rho}{8\rho}(G_{i}^+-G_{j}^-)^2 \right\}  \; ,
\end{align}
where $\{(\Gamma_i^+,G_i^+), i\geq 1\}$ and $\{(\Gamma_i^-,G_i^-), i\geq1\}$ are the points of two
independent point processes with mean measure $\frac12\rme^{-x} \rmd x \Phi(\rmd t)$.  If $\rho=1$,
the one dimensional case is recovered, but there is a discontinuity with the spherical case $\rho=0$
where the limiting distribution is Gumbel and the normalization is different.  Moreover, if
$\hat\mbx_n$ and $\check{\mbx}_n$ are the points such that $\|\hat\mbx_n -\check\mbx_n \| =
\diam2(\mbx)$, if $\hat\Theta_n$ and $\check\Theta_n$ are their respective angle such that
$\cos\hat\Theta_n>0$, $\cos\check\Theta_n<0$, then
$((\hat\Theta_n-\pi/4)/c_n,(\check\Theta_n-5\pi/4)/c_n)$ converges weakly to a pair
of~i.i.d.~Gaussian random variables with mean zero and variance $(1-\rho)/2\rho$.

\begin{figure}[h]

\centering 
\subfigure[$\rho=0.2$]
{\includegraphics[scale=0.4, clip=true,trim=0 70 0 80]{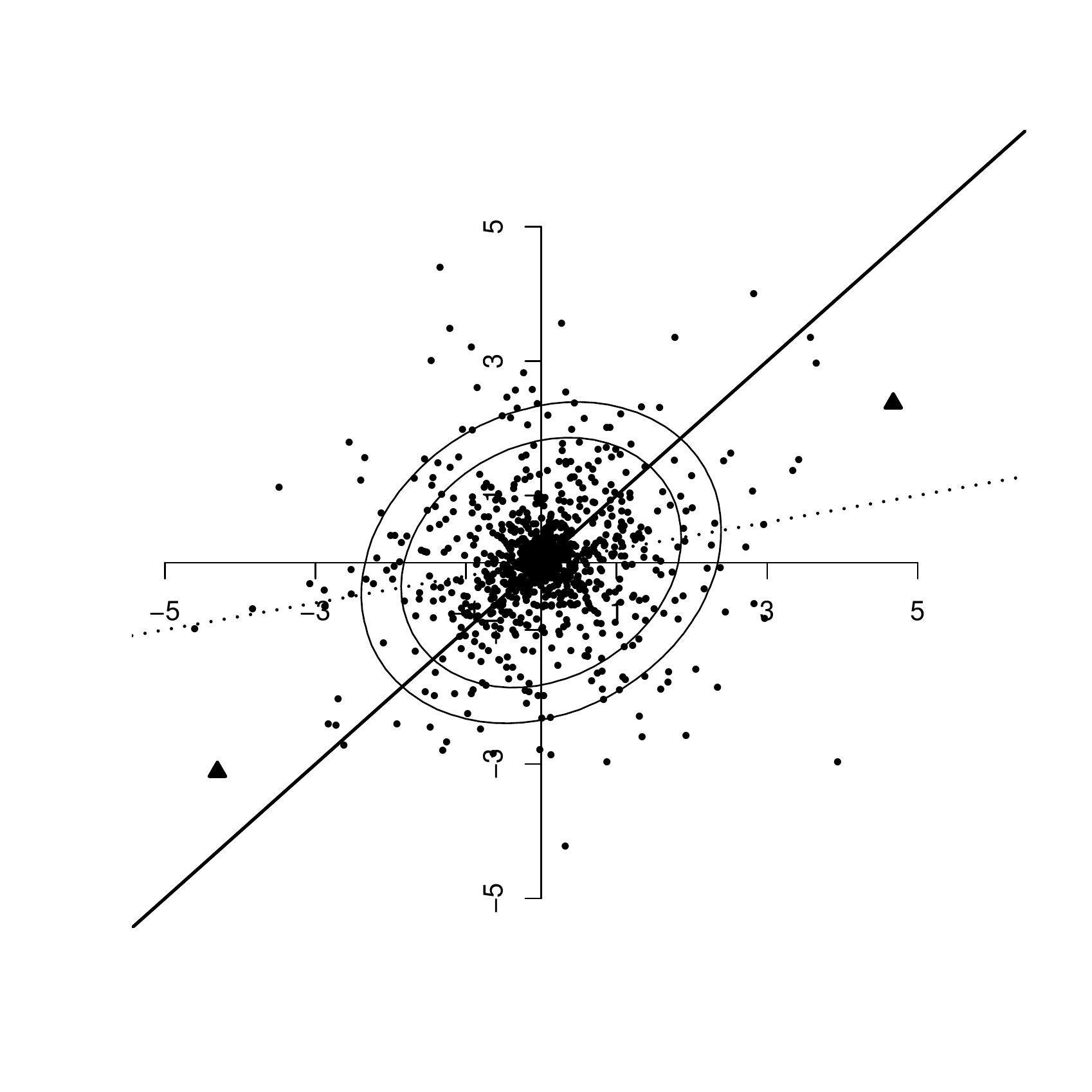}} 
\goodgap
\subfigure[$\rho=0.8$]
{\includegraphics[scale=0.4, clip=true,trim=0 70 0 80]{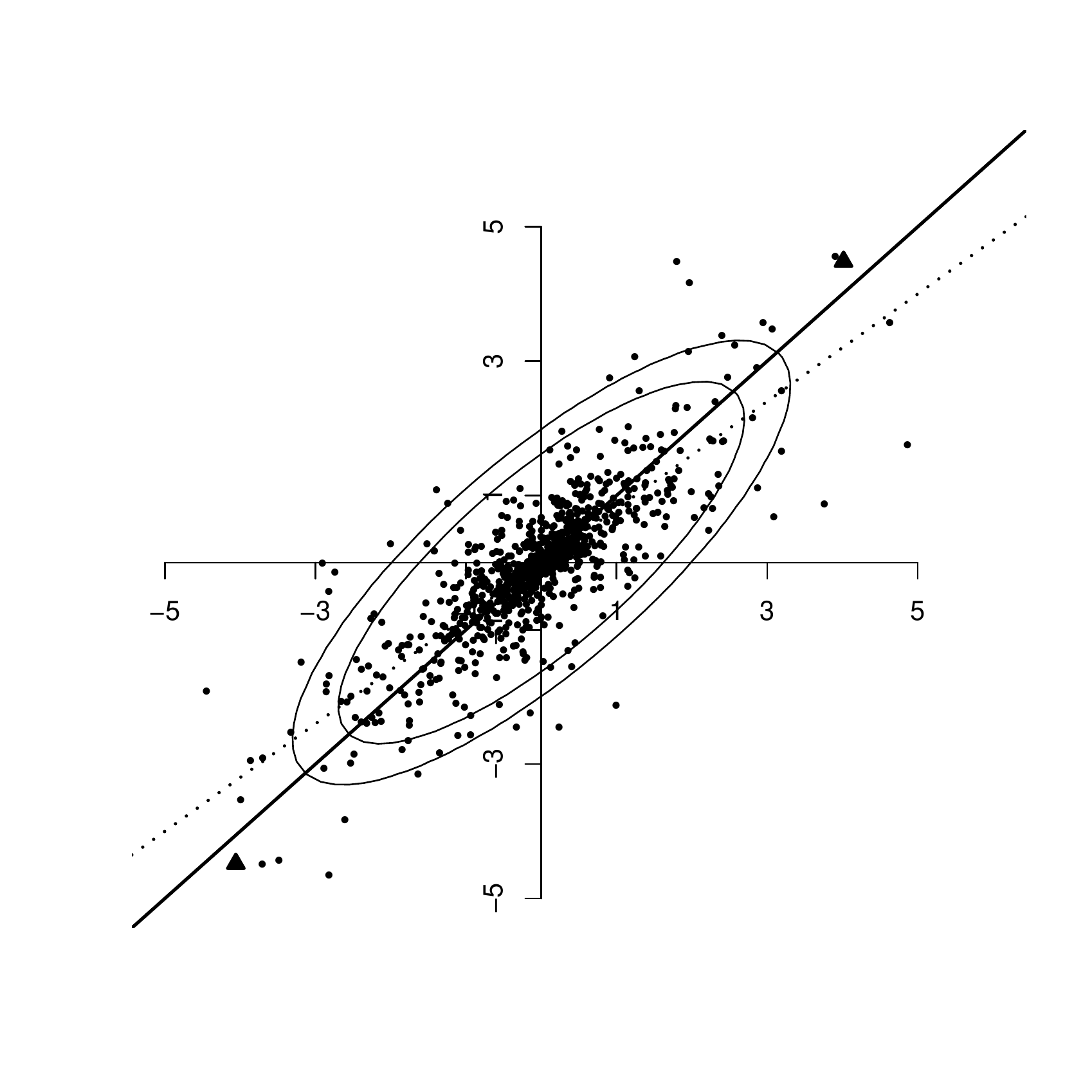}}
\vspace{-0.2cm}

\caption{\footnotesize{Two bidimensional Gaussian clouds. The black triangles are the points which realize
    the diameter. The black line is the diagonal and the dotted line is the regression line $y=\rho
    x$. The ellipses are the level lines of the density of the Gaussian distribution.}}
\label{fig:cloud2D_rho02}
\end{figure}

In Figure~\ref{fig:cloud2D_rho02} we show two sample clouds of size $1000$ of bivariate Gaussian
variables with correlation $\rho=0.2$ and $\rho=0.8$. The rate of convergence to the diagonal is
$O(\log n)$. In Figure~\ref{fig:cdf-diam-d=2}, we show the empirical cumulative distribution
function (cdf) of the limiting distribution based on 500 replications of the diameter of a Gaussian
cloud (with correlation $\rho=0.2$) of size 100 000.  In simulations, the indices realizing the
maximum in~(\ref{eq:loi-diam-d2}) are often $i=1$ and $j=1$. This implies that the limiting
distribution of the diameter should be close to the distribution of the sum of two independent
Gumbel random variables minus the square of a Gaussian random variable.  We show this distribution
together with the empirical and theoretical cdf of the diameter in Figure~\ref{fig:cdf-diam-d=2}.

\begin{figure}[h!]
\centering
 \includegraphics[width=12cm,height=6cm]{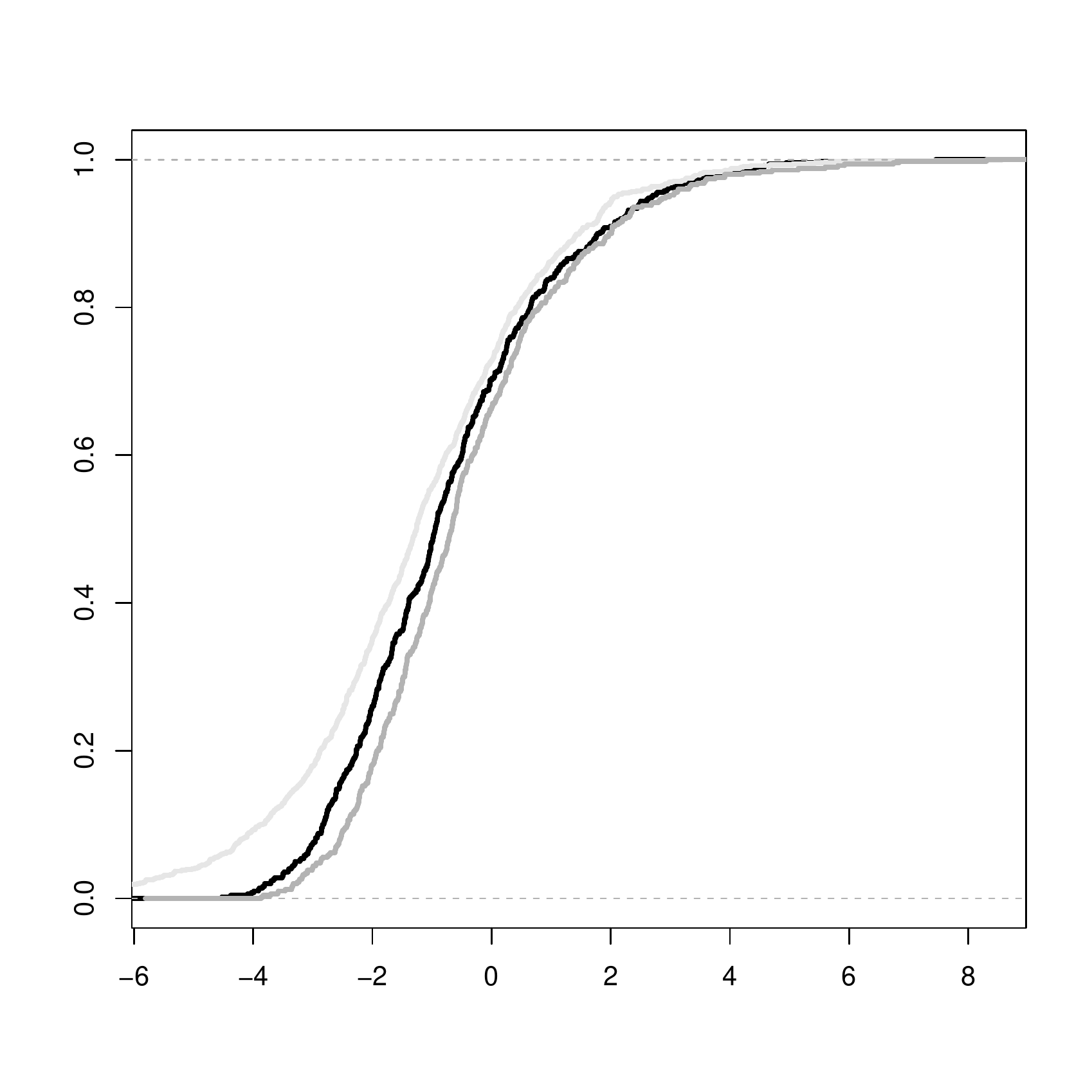}
 \caption{\footnotesize{Distribution function of the limiting distribution of the diameter of a
     bidimensional elliptical cloud. The black thick line is the (simulated) theoretic cdf; the
     thick gray line is the empirical cdf based on 500 clouds of 100 000 points. The thin gray line
     is the cdf of the sum of two independent Gumbel random variables with location parameter
     $\log2$.}}
 \label{fig:cdf-diam-d=2}
\end{figure}

\subsubsection{Proof of Theorem~\ref{theo:diam-k=1}}

Define the set $O_n = \{w\in\Rset^{d-1} \mid c_n^2\|w\|^2\leq 1\}$ and the function $f_n$ on
$\Rset\times\{-1,+1\}\times O_n\to \Rset^d$ by
\begin{align*}
  f_n (r,\epsilon,w) = (a_n+b_nr) (\epsilon \sqrt{1-c_n^2\|w\|^2},c_nw) \; .
\end{align*}

Define next the function $g_n$ on $\Rset\times\{-1,+1\}\times O_n\times
\Rset\times\{-1,+1\}\times O_n$ by
\begin{align*}
  g_n(r_1,\epsilon_1,w_1,r_2,\epsilon_2,w_2) = \frac{\|f_n(r_1,\epsilon_1,w_1)
    -f_n(r_2,\epsilon_2,w_2)\|-2a_n}{b_n} \; .
\end{align*}
Since $c_n\to0$, any $w\in\Rset^{d-1}$ is in $O_n$ for all large enough $n$. Then,  for any
$r_1,r_2>0$, for $\epsilon_1,\epsilon_2\in\{-1,1\}$ and any $w_1,w_2\in\Rset^{d-1}$,
\begin{align*}
  \lim_{n\to\infty} g_n(r_1,\epsilon_1,w_1,r_2,\epsilon_2,w_2) & = g(r_1,\epsilon_1,w_1,r_2,\epsilon_2,w_2) \\
  & =
  \begin{cases}
    -\infty & \mbox{ if } \epsilon_1\epsilon_2 = 1 \; , \\
    r_1+r_2 -\frac14 \sum_{q=2}^d (w_{1,q}-w_{2,q})^2 & \mbox{ if } \epsilon_1\epsilon_2 = - 1 \; .
  \end{cases}
\end{align*}
The convergence is locally uniform.  Moreover
\begin{align*}
  \frac{\diam2(\mby)-2a_n}{b_n} = \max_{1 \leq i < j \leq n}
  g_n(P_{n,i},P_{n,j}) \; .
\end{align*}
We want to conclude that the limiting distribution of $(\diam2(\mby)-2a_n)/b_n$ is $\max_{1 \leq i <
  j } g(P_i,P_j)$ (where the points $P_i$ are defined in~(\ref{eq:def-Pi})) by a continuous mapping
argument, but some care is needed.

Define $M_n^+ = \max\{\|\mby_i\| \mid \Theta_{1,i} > 0\}$ and $M_n^- = \max\{\|\mby_i\| \mid
\Theta_{1,i} < 0\}$. Let $\mby_n^+$ and $\mby_n^-$ be the points such that $M_n^+ = \|\mby_n^+\|$
and $M_n^- = \|\mby_n^-\|$. Then, by definition of the diameter, we have
\begin{align*}
\|\mby_n^+ - \mby_n^-\| \leq \diam2(\mby) \leq M_n^+ + M_n^- \; .
\end{align*}

Define $A_n = \|\mby_n^+ - \mby_n^-\| - M_n^+-M_n^-$. This yields the following lower and upper
bounds for the diameter:
\begin{align}
  \label{eq:encadrement-diametre}
  M_n^+ + M_n^- - A_n \leq \diam2(\mby) \leq M_n^+ + M_n^- \; .
\end{align}

As a corollary of the point process convergence, we obtain that
\begin{align*}
  \left(\frac{M_n^+-a_n}{b_n} ,\frac{M_n^--a_n}{b_n} ,\frac{A_n}{b_n} \right)\convweak \left(
    \Gamma_1^+, \Gamma_1^-, \frac14 \sum_{q=2}^d \tau_q^2(G_{1,q}^+ - G_{1,q}^-)^2 \right) \; .
\end{align*}

The bounds~(\ref{eq:encadrement-diametre}) imply that the diameter is achieved by a pair of points
$(\hat\mby_n,\check\mby_n)$ such that
\begin{align*}
  \|\hat\mby_n\| \wedge \|\check\mby_n\| \geq M_n^+ \wedge M_n^- - A_n \; .
\end{align*}
Indeed otherwise,
\begin{align*}
  \diam2(\mby) & = \|\hat\mby_n-\check\mby_n\| \leq \|\hat\mby_n\| + \|\check\mby_n\| < M_n^+ \wedge M_n^- - A_n +  M_n^+\vee M_n^-  \\
  & \leq M_n^+ + M_n^- - A_n \leq \diam2(\mby) \; ,
\end{align*}
which is a contradiction. This implies that
\begin{align*}
  \frac{\diam2(\mby) -2a_n}{b_n} = \max_{x,y \in E_n} g_n(x,y) \; ,
\end{align*}
where $E_n$ is the set of points of $N_n$ whose first component is at least equal to $(M_n^+ \wedge
M_n^--a_n)/b_n - A_n/b_n$,~i.e.
\begin{align*}
  E_n = \{P_{n,i} \mid \|X_i\| \geq M_n^+ \wedge M_n^- - A_n\} \; .
\end{align*}

Since by definition $\mby_n^+$ and $\mby_n^-$ belong to $E_n$, it obviously holds that
\begin{align*}
  \frac{\diam2(\mby) -2a_n}{b_n} \geq \max_{(x,y) \in E_n^+\times E_n^-} g_n(x,y) \geq
  g_n(P_n^+,P_n^-) \; ,
\end{align*}
where $E_n^+$ and $E_n^-$ are the points of $E_n$ whose second component is positive or negative,
respectively, and $P_n^\pm$ is the point of $E_n^\pm$ with the largest first component,
i.e.~$\mby_n^\pm$.

The convergence of the points of $N_n$ suitably numbered to those of $N$ imply that the sets $E_n^+$
and $E_n^-$ converge to the sets $E^+$ and $E^-$ of points of $N^+$ and $N^-$ defined by
\begin{align*}
  E^\pm = \bigg\{ P_i^\pm \mid \Gamma_i^\pm \geq \Gamma_1^+\wedge \Gamma_1^- - \frac14 \sum_{q=2}^d
  \tau_q^2(G_{1,q}^+-G_{1,q}^-)^2 \bigg\} \; .
\end{align*}
The sets $E^+$ and $E^-$ are almost surely finite since the points $\Gamma_i^\pm$ are only finitely
many in any interval~$(x,\infty)$.  This implies that the cardinals of the sets $E_n^\pm$ are
constant for large enough $n$.  By Skorohod's representation theorem
\cite[Theorem~3.30]{kallenberg:2002}, we may moreover assume that the points of $E_n^\pm$ converge
almost surely to those of $E^\pm$.

Since $g_n$ converges uniformly to $g$ on compact sets of $\Rset\times\{1\}\times\Rset^{d-1} \times
\Rset\times\{-1\}\times\Rset^{d-1}$ and since $P_{n}^\pm$ converge to $P_1^\pm$, $g_n(P_n^+,P_n^-)$
converges to $g(P_1^+,P_1^-)$ which is finite.  On the other hand, the points of $(E_n^+\times
E_n^+) \cup (E_n^-\times E_n^-)$ are all included in a fixed compact set and thus $\lim_{n\to\infty}
\max_{(x,y)\in(E_n^+\times E_n^+) \cup (E_n^-\times E_n^-)} g_n(x,y)=-\infty$. This implies that for
$n$ large enough,
\begin{align*}
  \max_{(x,y) \in (E_n^+\times E_n^+) \cup (E_n^-\times E_n^-)} g_n(x,y) \leq g_n(P_n^+,P_n^-) \; .
\end{align*}
We conclude that
\begin{align*}
  \frac{\diam2(\mby)-2a_n}{b_n} = \max_{x,y \in E_n} g_n(x,y) = \max_{(x,y)\in E_n^+\times
    E_n^-}g_n(x,y) \; .
\end{align*}

We can now apply a continuous mapping argument, since $g_n$ converges uniformly to $g$ on compact
sets of $\Rset\times\{1\}\times\Rset^{d-1} \times \Rset\times\{-1\}\times\Rset^{d-1}$. We obtain
\begin{align*}
  \frac{\diam2(\mby)-2a_n}{b_n} \convweak \max_{({x},{y}) \in
    E^+\times E^-} g(x,y) \; .
\end{align*}

To see that this is identical to (\ref{eq:limit-diam-k=1}), note that if
\begin{align*}
  \Gamma_i^+\wedge \Gamma_j^- < \Gamma_1^+\wedge \Gamma_1^- - \frac14\sum_{q=2}^d
  \tau_q^2(G_{1,q}^+-G_{1,q}^-)^2 \; ,
\end{align*}
then
\begin{align*}
  \Gamma_i^+ +\Gamma_j^- & - \frac14 \sum_{q=2}^d \tau_q^2
  (G_{i,q}^+-G_{j,q}^-)^2 \\
  & \leq\Gamma_i^++\Gamma_j^- \leq \Gamma_1^+\wedge \Gamma_1^- - \frac14
  \sum_{q=2}^d \tau_q^2(G_{1,q}^+-G_{1,q}^-)^2 + \Gamma_1^+ \vee \Gamma_1^- \\
  & = \Gamma_1^+ + \Gamma_1^- - \frac14 \sum_{q=2}^d \tau_q^2(G_{1,q}^+-G_{1,q}^-)^2 \; .
\end{align*}
This proves that the maximum of $g$ over all pairs of points of $N^+$ and $N^-$ is actually obtained
over the pairs of $E^+\times E^-$.

\subsection{Case $k>1$: multiple maximum eigenvalue}

If $k>1$, as in \cite{jammalamadaka:janson:2012}, a strengthening of domain of attraction condition
is needed to prove the result. Since an auxiliary function $\psi$ can be chosen differentiable and
such that $\lim_{x\to\infty} \psi'(x)=0$, it always holds that
$\lim_{x\to\infty}\psi(x+\psi(x)t)/\psi(x) = 1$ locally uniformly with respect to~$t\in\Rset$. We
must strengthen this uniformity as follows.

\begin{assumption}
  \label{hypo:strengthening-psi}
  For any positive function $\ell$ such that $\ell(x)\to\infty$ and $\ell(x) = O(\log(x/\psi(x)))$
  as $x\to\infty$,
  \begin{align}
    \label{eq:strengthening-psi-square}
    \lim_{x\to\infty} \ell(x) \sup_{|t|\leq \ell(x)}\left|\frac{\psi(x+\psi(x)t)}{\psi(x)} -
      1\right| = 0 \; .
  \end{align}
\end{assumption}

This assumption is satisfied by all usual distributions, such as the Weibull, Gaussian, exponential
or log-normal distributions. An important consequence is that the quantile of order $1-1/n$ of
$\|\mbx\|$ and $T$ can be related. Recall from Theorem~\ref{theo:norm-ddim} that, as $x\to\infty$,
\begin{align}
  \pr(\|\mbx\| > x\sqrt{\lambda_1}) \sim D_k \left( \frac{\psiT(x)}{x}\right)^{(d-k)/2} \pr(T>x)
  \;, \label{eq:equiv-survival-norm}
\end{align}
with
\begin{align*}
  D_k = \frac{\Gamma(d/2)}{\Gamma(k/2)} 2^{(d-k)/2} \left(\prod_{q=k+1}^d
    \frac{\lambda_1}{\lambda_1-\lambda_q} \right)^{1/2} \; .
\end{align*}
Let $a_n^T$ be such that $\pr(T>a_n^T) \sim 1/n$ and set $b_n^T=\psiT(a_n^T)$.  Define the sequence
$\{a_n\}$ by
\begin{align}
  \label{eq:quantile-relation}
  a_n = \sqrt{\lambda_1} \; a_n^T - \sqrt{\lambda_1} b_n^T \left(\frac{d-k}2 \log\frac{a_n^T}{b_n^T}
    - \log D_k \right) \; .
\end{align}
Then $\pr(\|\mbx\|>a_n)\sim 1/n$.  This is a consequence of the
equivalence~(\ref{eq:equiv-survival-norm}) and Lemma~\ref{lem:strengthening-log}. Let thus $a_n$ be
defined as in~(\ref{eq:quantile-relation}) and define $b_n=\psiA(a_n)$ and
\begin{align*}
  d_n = \frac{k-1}2 \log\frac{a_n}{b_n} - \log\log\frac{a_n}{b_n} - \log C_k \;
\end{align*}
with 
\begin{align}
  \label{eq:def-Ck}
  C_k = (2d-k-1) 2^{k-4} \pi^{-1/2} \Gamma(k/2) \left(\prod_{q=k+1}^d \frac{\lambda_1}
    {\lambda_1-\lambda_q} \right)^{-1/2} \; .
\end{align}

\begin{theorem}
  \label{theo:diam-k>1}
  Let $\{\mbx_i,i\geq1\}$ be a sequence of i.i.d.~random vectors with the same distribution as
  $\mbx$ and let the assumptions of Theorem~\ref{theo:norm-ddim} hold with $k\geq2$.  If moreover
  Assumption~\ref{hypo:strengthening-psi} holds, then for all $z\in\Rset$,
  \begin{align}
    \label{eq:convergence-gumbel}
    \lim_{n\to\infty} \pr \left( \frac{\diam2(\mbx)-2a_n}{b_n} + d_n \leq z \right)=
    \rme^{-\rme^{-z}} \; .
  \end{align}
\end{theorem}

\paragraph{Comments}
\begin{itemize}
\item In the spherical case $k=d$, we recover \cite[Theorem~1.1]{jammalamadaka:janson:2012} and the
  constant $c_d$ therein is equal to the constant $C_k$ in~(\ref{eq:def-Ck}) (taking the product
  over an empty set of indices to be equal to 1).
\item We actually prove slightly more than the convergence~(\ref{eq:convergence-gumbel}). The proof
  can be used to check the conditions of Kallenberg's Theorem (see
  e.g.~\cite[Proposition~3.22]{resnick:1987}) which prove that the point process
  \begin{align*}
    \sum_{1 \leq i < j \leq n} \delta_{\frac{\|\mbx_i-\mbx_j\|-2a_n}{b_n}+d_n}
  \end{align*}
  converges to a Poisson point process with mean measure $\rme^{-x}\rmd x$ on
  $(-\infty,\infty]$. This result might be used for instance to derive the asymptotic distribution
  of the order statistics of the interpoint distances.
\end{itemize}

\subsubsection*{Proof of Theorem~\ref{theo:diam-k>1}}
\label{sec:proof-theo-diamk>1}

The proof is nearly the same as the proof of \cite[Theorem~1.1]{jammalamadaka:janson:2012}. We prove
the convergence of a $U$-statistic of indicators to a Poisson random variable.  The difference lies
in added technicalities due to the coordinates of the vector corresponding to the smaller
eigenvalues which have to be integrated out. In more precise terms, as in the proof of
Theorem~\ref{theo:diam-k=1}, we work with vague convergence of measures rather than weak
convergence.

Define $s_n = \frac12 \log d_n$ and
\begin{align*}
  S_n(z) & = \sum_{1 \leq i < j \leq n} \1{\{\|\mby_i - \mby_j\| > 2a_n - b_n d_n + b_nz\}} \; .
\end{align*}

Since $\pr(\diam2(\mby)>2a_n-b_nd_n+b_nz)=\pr(S_n(z)=0)$, it suffices to prove that for all
$z\in\Rset$, $S_n(z)$ converges weakly to a Poisson random variable with mean $\rme^{-z}$. For
technical reasons, as in\cite{jammalamadaka:janson:2012}, we must truncate the sum defining
$S_n(z)$.  Define
\begin{align*}
  S_n'(z) & = \sum_{1 \leq i < j \leq n} \1{\{\|\mby_i - \mby_j\| > 2a_n -b_n d_n + b_nz\}} \1{\{
    T_i \vee T_j \leq a_n^T + b_n^T s_n\}} \; .
\end{align*}
In words, we restrict the sum to the indices of vectors whose norm is not too large, hence not too
small either, since their distance must be large. Note that $S_n(z)\ne S_n'(z)$ implies that there
is at least one index $i$ such that $T_i>a_n^T+b_n^Ts_n$.  Since $s_n\to\infty$, this implies that
for any $A>0$,
\begin{align*}
  \limsup_{n\to\infty} \pr(S_n(z)\ne S_n(z')) & \leq \limsup_{n\to\infty}  n\pr(T>a_n^T+b_n^T s_n) \\
  & \leq \limsup_{n\to\infty} n\pr(T>a_n^T+b_n^T A) = \rme^{-A} \; .
\end{align*}
Since $A$ is arbitrary, this proves that for all $z\in\Rset$,
\begin{align*}
  \lim_{n\to\infty} \pr(S_n(z)\ne S'_n(z)) = 0 \; .
\end{align*}
This in turn implies that we only need to prove that $S'_n(z)$ converges weakly to a Poisson random
variable with mean $\rme^{-z}$.  This convergence is obtained by applying the criterion of
\cite[Theorem 3.1 and Remark 3.4]{jammalamadaka:janson:1986}.
\begin{lemma}
  \label{lem:poisson-convergence}
  Under the Assumptions of Theorem~\ref{theo:diam-k>1},
  \begin{align}
    \label{eq:poisson-convergence-esp}
    \lim_{n\to\infty} & \frac{n^2}2 \pr(\|\mby_1 - \mby_2\| > 2a_n - b_n d_n +
    b_nz \; ; \; T_2 \vee T_2 \leq a_n^T + b_n^T s_n) = \rme^{-z} \; ,  \\
    \lim_{n\to\infty} & n^3 \pr(\|\mby_1 - \mby_2\| \wedge \|\mby_1 - \mby_3\| > 2a_n - b_n d_n +
    b_nz \; ; \; T_2 \vee T_2 \leq a_n^T + b_n^T s_n) = 0 \;
    . \label{eq:poisson-convergence-neglige}
  \end{align}
\end{lemma}
The convergences~(\ref{eq:poisson-convergence-esp}) and~(\ref{eq:poisson-convergence-neglige}) imply
that $S'_n(z)$ converges weakly to a Poisson distribution with mean $\rme^{-z}$ and this concludes
the proof of Theorem~\ref{theo:diam-k>1}.

The proof of Lemma~\ref{lem:poisson-convergence} consists mainly in checking the vague convergence
of certain measures and then strengthening this convergence to weak convergence by bounded
convergence arguments.  The requested bounds are obtained by means of
Assumption~\ref{hypo:strengthening-psi} which is slightly stronger than the assumption of uniformity
used in \cite[Theorem~1.1]{jammalamadaka:janson:2012}, but is satisfied for all usual distributions.
Apart from these arguments, the proof follows the same lines as the proof of
\cite[Theorem~1.1]{jammalamadaka:janson:2012}.  In view of their tedious technical nature, this
proof is postponed to Section~\ref{sec:proof-lemmas}.

Let us note that as a by-product of the proof, we obtain in
Lemma~\ref{lem:convergence-jointe-cosinus} the convergence of the cosine of the angle between two
vectors $\mby_1$ and $\mby_2$ and of the components corresponding to the smaller eigenvalues, given
that their distance is large and their norm is large, but not too large (this is quantified in the
definition of $S'_n(z)$).  This parallels the convergence proved in Theorem~\ref{theo:norm-ddim},
but we do not explicitly use it in the proof of Theorem~\ref{theo:diam-k>1}. It may eventually prove
to be of interest for some other problem.

\section{The  $l^q$ norm of a random spherical vector}
\label{sec:lq-spherical}
In this Section, the localization principle will be used to answer another question raised
in~\cite{jammalamadaka:janson:2012}, namely, the asymptotic behavior of the $l^q$ diameter of a
cloud of spherical random vectors in dimension $d\geq2$.  Define the $l^q$ norm of a vector
$x\in\Rset^d$ by
\begin{align*}
  \|x\|_q = \left(|x_1|^q+\cdots + |x_d|^q \right)^{1/q} \; .
\end{align*}
For $d\geq2$ and $q\geq1$, $q\ne2$, the maximum of the $l^q$ norm is achieved on the $l^2$ sphere
$\sphere{d-1}$ at isolated points. Specifically,
\begin{itemize}
\item if $q\in[1,2)$, then $\max_{w\in\sphere{d-1}} \|w\|_q=d^{1/q-1/2}$; it is achieved at the
  $2^d$ ``diagonal'' points $(\pm d^{-1/2},\dots,\pm d^{-1/2})$.
\item if $q\in(2,\infty)$, then $\max_{w\in\sphere{d-1}} \|w\|_q=1$;
  the maximum is achieved at the $2d$ intersections of the axes with $\sphere{d-1}$.
\end{itemize}
Therefore, the localization phenomenon will occur. A spherical vector whose norm is large must be
close to the direction of one of these maxima, and the diameter will be achieved by points which are
nearly diametrically opposed along one of these directions.

We consider a spherically distributed random vector, i.e.~$\mbx = T\mbw$ where $T$ and $\mbw$ are
independent and $\mbw$ is uniform on $\sphere{d-1}$. Let $\{\mbx_i,i\geq1\}$ be a sequence of
i.i.d.~vectors with the same distribution as $\mbx$.  Define
\begin{align*}
  \maxq(\mbx) = \max_{1 \leq i \leq n} \|\mbx_i\|_q \; , \ \ \diamq(\mbx) = \max_{1 \leq i < j \leq
    n} \|\mbx_i - \mbx_j\|_q \; .
\end{align*}
The behavior of $\|\mbx\|_q$ differs only by constants for $q\in[1,2)$ and for $q>2$, whereas the
diameter has two very different behavior if $q\in[1,2)$ and $q>2$. Therefore, we study these two
cases separately. We start with the case $q>2$ which is somewhat easier.

\subsection{Case $q > 2$}
For $q>2$, the maximum of the $l^q$ norm on the $l^2$ sphere is 1 and is achieved at the $2d$
intersections of the sphere with the axes.  We will see that the localization of the vectors with
large norms occurs at a very fast rate, and therefore the diameter behaves asymptotically as in the
one dimension case.

For $i=1,\dots,d$, define $\Delta_i = \{x\in\Rset^d \mid x_i > \max_{1 \leq j\leq d, j\ne i}
|x_j|\}$ and $\Delta_{-i} = -\Delta_i$. Then $\cap_{1 \leq i,-i \leq d} \Delta_i =\emptyset$ and
$\cup _{1 \leq i, -i \leq d} \overline{\Delta_i} = \Rset^d$ (where $\bar{A}$ is the closure of
a set $A\subset\Rset^d$).  Define $\phi(x) = \sqrt{\psi(x)/x}$.
\begin{theorem}
  \label{theo:norm-lq-q>2}
  Let $\mbx = T\mbw$ where $T$ and $\mbw$ are independent, $\mbw$ is uniform on $\sphere{d-1}$
  and~$T$ satisfies~(\ref{eq:domaine}).  For $q\in(2,\infty]$,
  \begin{align}
    \label{eq:norm-lq-q>2}
    \pr(\|\mbx\|_q > x) \sim \frac{2^{(d-1)/2} \, d \, \Gamma(d/2)} {\Gamma(1/2)} \phi^{d-1}(x)
    \pr(T>x) \; .
  \end{align}
  Moreover, conditionally on $\|\mbx\|_q>x$ and $\mbx\in\Delta_1$, as $x\to\infty$,
  \begin{align*}
    \left( \frac{\|\mbx\|_q-x}{\psi(x)} , \frac{1-W_1}{\phi^2(x)}, \frac{W_2}{\phi(x)} ,\dots,
      \frac{W_d}{\phi(x)} \right) \convweak (E, \frac12(G_2^2+\dots+G_d^2), G_2,\dots,G_d)
  \end{align*}
  where $E$ is an exponential random variable with mean $1$ and $G_2,\dots,G_d$ are i.i.d.~standard
  Gaussian random variables, independent of $E$.
\end{theorem}

\paragraph{Comments} If $T^2$ has a $\chi^2$ distribution with $d$ degrees of freedom, then $\mbx$ is
a standard $d$-dimensional Gaussian vector and \Cref{theo:norm-lq-q>2} is a particular case of
\cite[Theorem~1 and Example~1]{hashorva:korshunov:piterbarg:2013}. In that case, $\pr(T>x) \sim (1/2)^{d/2}
\Gamma^{-1}(d/2)$, $\phi(x)=1/x$ and the equivalent~(\ref{eq:norm-lq-q>2}) yields
\begin{align*}
  \pr(\|\mbx\|_q^q>x) \sim \frac{2d}{\sqrt{2\pi}} x^{-1/q} \rme^{-\frac12 x^{2/q}} \; .
\end{align*}
The tail depends on $d$ only in the constant but not in the exponent. This is expected since
$\|\mbx\|_q^q$ is the sum of $d$ independent random variables with subexponential tails. Hence, by
definition of subexponentiality, the this sum is tail equivalent to $d$ times the tail of one
variable. This is specific to the Gaussian case, since otherwise the components of $\mbx$ are not
independent.

\begin{proof}[Proof of~\Cref{theo:norm-lq-q>2}]
  If $\mbw$ is uniformly distributed on $\sphere{d-1}$, then the distribution of $W_1$ has the
  density $\beta_d (1-s^2)^{(d-3)/2}$ on $[-1,1]$ with $\beta_d =
  \frac{\Gamma(d/2)}{\Gamma((d-1)/2)\Gamma(1/2)}$.  Define $\tilde\mbw =
  (1-W_1^2)^{-1/2}(W_2,\dots,W_d)$. By Lemma~\ref{lem:conditioning-uniform-sphere}, $\tilde\mbw$ is
  uniformly distributed on $\sphere{d-2}$ and independent of $W_1$.
  Let $f$ be continuous with compact support in $\Rset^d$ and define
  the function $k_x$ on $[0,\phi^{-2}(x)]\times\Rset$ by
  \begin{multline*}
    k_x(u,z) \\
    = \frac1{\pr(T>x)}\esp \left[ f\left(\frac{1-u}{\phi^2(x)}, \frac{\sqrt{1-u^2}}{\phi(x)}
        \tilde\mbw\right) \1{\{T > \frac{x + \psiT(x)z} {\{u^q + ((1-u^2)^{q/2}
          \|\tilde\mbw\|_q^q\}^{1/q}} \}} \1{\{u > \sqrt{1-u^2} \max_{2 \leq i \leq d}
        |\tilde{W}_i|\}} \right] \; ,
  \end{multline*}
  for $2<q<\infty$ or
  \begin{align*}
    k_x(u,z) = \frac1{\pr(T>x)} \esp \left[ f\left(\frac{1-u}{\phi^2(x)},
        \frac{\sqrt{1-u^2}}{\phi(x)} \tilde\mbw\right) \1{\{T > \frac{x + \psiT(x)z} {u \vee
          \sqrt{1-u^2} \|\tilde\mbw\|_\infty}\}} \1{\{u > \sqrt{1-u^2} \max_{2 \leq i \leq d}
        |\tilde{W}_i|\}} \right] \; ,
  \end{align*}
  for $q=\infty$.  Then the following convergence holds, locally uniformly on
  $[0,\infty)\times\Rset$,
  \begin{align*}
    \lim_{x\to\infty} k_x(1-\phi^2(x)v/2,z) = \rme^{-z} \,\rme^{-v/2}\, \esp \left[ f\left(v/2,
        \sqrt{v} \tilde\mbw\right) \right] \; .
  \end{align*}
  This yields, for $f$ continuous and compactly supported on $\Rset^d$,
  \begin{align*}
    \esp & \left[ f\left(\frac{1-W_1}{\phi^2(x)}, \frac{(W_2,\dots,W_d)}{\phi(x)} \right)
      \1{\{\|\mbx\|_q > x + \psiT(x)z\}} \1{\Delta_1}(\mbw)\right] \\
    & = \beta_d \pr(T>x) \, \int_{0}^1 k_x(u,z) (1-u^2)^{(d-3)/2} \, \rmd u \\
    & = \beta_d \phi^{d-1}(x) \pr(T>x) \, \int_{0}^{2\phi^{-2}(x)} k_x(1-\phi^2(x)v/2,z)
    (v-\phi^2(x)v^2/4)^{(d-3)/2} \, \rmd v/2    \\
    & \sim \beta_d \phi^{d-1}(x) \pr(T>x) \, \rme^{-z} \int_0^\infty \esp \left[ f\left(v/2,
        \sqrt{v} \tilde\mbw\right) \right]  \,  v^{(d-3)/2} \, \rme^{-v/2} \, \rmd v/2 \\
    & = \beta_d \phi^{d-1}(x) \pr(T>x) \, \rme^{-z} \, \Gamma((d-1)/2) 2^{(d-3)/2} \, \esp \left[
      f\left(R^2/2, R \tilde\mbw \right) \right] \; ,
  \end{align*}
  where $R^{2}$ has a $\chi^2$ distribution with ${d-1}$ degrees of freedom and is independent
  of $\tilde\mbw$.  This implies that $\sqrt{R} \tilde\mbw$ is a $(d-1)$ dimensional standard
  Gaussian vector. Equivalently, $(R^2/2,R\tilde\mbw)$ can be expressed as
  $(\frac12(G_2^2+\cdots+G_d^2),G_2,\dots,G_d)$, where $G_2,\dots,G_d$ are i.i.d.~standard Gaussian
  random variables. This yields, for $f$ continuous and compactly supported on $\Rset^d$,
  \begin{multline*}
    \lim_{x\to\infty} \frac1{\phi^{d-1}(x)\,\pr(T>x)} \esp \left[ f\left(\frac{1-W_1}{\phi^2(x)},
        \frac{(W_2,\dots,W_d)}{\phi(x)} \right) \1{\{\|\mbx\|_q > x + \psiT(x)z\}}
      \1{\Delta_1}(\mbw)\right] \\
    = \frac{2^{(d-3)/2} \Gamma(d/2)} {\Gamma(1/2)}
    \esp\left[f\left(\tfrac12(G_2^2+\cdots+G_d^2),G_2,\dots,G_d\right) \right] \rme^{-z} \; .
  \end{multline*}
  The last step is to extend the convergence to bounded continuous functions. This is done as in the
  proof of Theorem~\ref{theo:norm-ddim}, using the bound~(\ref{eq:borne-uniforme-avec-z}). Summing
  these equivalent over the $2d$ regions $\Delta_i$ yields~(\ref{eq:norm-lq-q>2}).
\end{proof}

Define
\begin{align} 
  \label{eq:def-mbu}
  \mbu = \frac{\mbx}{\|\mbx\|_q} = \frac{\mbw}{\|\mbw\|_q} \; .
\end{align}
\begin{corollary}
  \label{coro:excess-U}
  Under the assumptions of Theorem~\ref{theo:norm-lq-q>2}, as $x\to\infty$, conditionally on
  $\|\mbx\|_q>x$ and $\mbx\in\Delta_1$,
  \begin{align}
    \label{eq:convergence-degeneree}
    \left( \frac{\|\mbx\|_q-x}{\psi(x)} , \frac{1-U_1}{\phi^2(x)}, \frac{U_2}{\phi(x)} ,\dots,
      \frac{U_d}{\phi(x)} \right) \convweak (E, 0 , G_2,\dots,G_d) \; ,
  \end{align}
  where $E$ is an exponential random variable with mean $1$ and $G_2,\dots,G_d$ are i.i.d.~standard
  Gaussian random variables, independent of $E$.
\end{corollary}

\begin{proof}
  Define $R_x = (1 - W_1)/\phi^2(x)$. Conditionally on $\|\mbx\|_q>x$ and $\mbx\in\Delta_1$,
  $R_x\convweak R$, hence $R_x=O_P(1)$. Thus, conditionally on $\|\mbx\|_q>x$ and $\mbx\in\Delta_1$,
  \begin{align*}
    \|\mbw\|_q = 1 - \phi^2(x) R_x + o_P(\phi^2(x)) = 1 + O_P(\phi^2(x)) \; .
  \end{align*}
  This yields
  \begin{align*}
    \left( \frac{\|\mbx\|_q-x}{\psi(x)} , \frac{U_2}{\phi(x)},\dots,\frac{U_d}{\phi(x)}\right) & =
    \left( \frac{\|\mbx\|_q-x}{\psi(x)} , \frac{W_2}{\phi(x)\|\mbw_2\|}
      , \dots, \frac{W_d}{\phi(x)\|\mbw_d\|} \right) \\
    & = \left( \frac{\|\mbx\|_q-x}{\psi(x)} , \frac{W_2}{\phi(x)\{1 + O_P(\phi^2(x))\}}, \dots,
      \frac{W_d}{\phi(x)\{1 + O_P(\phi^2(x))\}}\right)    \\
    & \convweak (E,G_2,\dots,G_d) \; .
  \end{align*}
  Moreover,
  \begin{align*}
    1 - U_1 = 1- \frac{W_1}{\|\mbw\|_q} & = 1 - \frac{1 - \phi^2(x) R_x}{1 - \phi^2(x) R_x +
      o_P(\phi^2(x)) } = o_P(\phi^2(x)) \; .
  \end{align*}
This yields~(\ref{eq:convergence-degeneree}).
\end{proof}
The degeneracy with respect to the second variable in the convergence
(\ref{eq:convergence-degeneree}) is the key to the behavior of the diameter in this case. 
Let $a_n$ be the $1-1/n$ quantile of the distribution of $\|\mbx\|_q$ and $b_n = \psiT(a_n)$.
\begin{theorem}
  Let $\{\mbx_i,i\geq1\}$ be a sequence of i.i.d.~random vectors with the same distibution as $\mbx$
  which satisfies the assumptions of Theorem~\ref{theo:norm-lq-q>2}. Then, for $q\in(2,\infty]$,
  \begin{align}
     \label{eq:diamq-spherique-q>2}
     \frac{\diamq(\mbx) - 2a_n}{b_n} \convweak \; \max_{1 \leq i \leq d} (\Gamma_i^++\Gamma_i^-) \;     ,
   \end{align}
   where $\Gamma_i^+$ and $\Gamma_i^-$, $1\leq i \leq d$ are independent Gumbel random variable with
   location parameter $\log2d$.
\end{theorem}

\begin{proof}
  With probability tending to one, the diameter will be achieved by a pair of points in two
  symmetric regions $\Delta_i$ and $\Delta_{-i}$.

  For $j=1,\dots,d$, define the points
  \begin{align*}
    P_{n,i,j}^+ = \left( \tfrac{\|\mbx_i\|_q-a_n}{b_n}, \tfrac{\mbu_i-(0,\dots,1,\dots,0)}{c_n}
    \right) \; , P_{n,i,j}^- = \left( \tfrac{\|\mbx_i\|_q-a_n}{b_n},
      \tfrac{\mbu_i-(0,\dots,-1,\dots,0)}{c_n} \right)
  \end{align*}
  (where the $\pm1$ is on the $j$-th position) and the point processes $N_{n,j}^\pm =
  \sum_{i=1}^\infty \delta_{P_{n,i,j}^\pm} \1{\{\mbx_i\in\Delta_{\pm j}\}}$.
  Corollary~\ref{coro:excess-U} yields the point process convergence
  $\{N_{n,j}^+,N_{n,j}^-,j=1,\dots,d\} \convweak \{N_j^+,N_j^-,j=1,\dots,d\}$ with
  $N_j^\pm=\sum_{i=1}^\infty \delta_{P_i^\pm}$ and
  \begin{align*}
    P_i^\pm = (\Gamma_{i,j}^\pm,0,G_{i,j,2}^\pm,\dots,G_{i,j,d}^\pm) \; ,
  \end{align*}
  where $N_j^+$, $N_j^-$ are independent Poisson point processes,
  $\{\Gamma_{i,j}^\pm,i\geq1,j=1,\dots,d\}$ are the points of a Poisson point process with mean
  measure $\frac1{2d}\rme^{-x}\,\rmd x$, independent of the i.i.d.~standard Gaussian vectors
  $(G_{i,j,2}^\pm,\dots,G_{i,j,d}^\pm)$, $i\geq1$, $j=1,\dots,d$.

  For $j=1,\dots,d$, let $\mbx_{n,j}^+$ and $\mbx_{n,j}^-$ be the vectors of the sample with the
  largest norms in $\Delta_j$ and $\Delta_{-j}$, respectively. With probability tending to one, it
  holds that
  \begin{align*}
    \max_{1\leq j\leq d} \frac{ \|\mbx_{n,j}^+ - \mbx_{n,j}^-\|_q-2a_n}{b_n} \leq \frac{\diamq(\mbx)
      -2a_n}{b_n} \leq \max_{1\leq j\leq d} \left\{ \frac{ \|\mbx_{n,j}^+\|_q -a_n}{b_n} +
      \frac{\|\mbx_{n,j}^-\|_q-a_n}{b_n} \right\} \; .
  \end{align*}

  For each $j=1,\dots,d$, the sum of the rightmost terms inside the $\max$ converges weakly to
  $\Gamma_{j,1}^++\Gamma_{j,1}^-$.  Define $Z_{n,j}^\pm = (\|\mbx_{n,j}^\pm\|_q -a_n)/b_n$ and
  $\mbg_{n,j}^\pm = \left(\mbx_{n,j}^\pm/\|\mbx_{n,j}^\pm\|_q - (0,\dots,\pm1,\dots,0)\right)/c_n$. The point
  process convergence entails the following one:
  \begin{align*}
    (Z_{n,j}^+,Z_{n,j}^-,\mbg_{n,j}^+,\mbg_{n,j}^-) \convweak (\Gamma_1^+,\Gamma_1^-,
    0,\mbg_j^+,0,\mbg_j^-) \; , 
  \end{align*}
  where all components are independent and the components of $\mbg_j^\pm$ are standard Gaussian
  $(d-1)$ dimensional Gaussian vectors.  By Corollary~\ref{coro:excess-U},
  $c_n\mbg_{n,j,1}^\pm=o(c_n^2)$, thus,
  \begin{align*}
    \|\mbx_{n,j}^+ - \mbx_{n,j}^-\|_q & = a_n\left[ \left\{ (1+c_n^2Z_{n,j}^+)(1+c_n\mbg_{n,j,1}^+)+
        (1+c_n^2Z_{n,j}^-)(1-c_n\mbg_{n,j,1}^-)\right\}^q + O(c_n^q)\right] \\
    & = 2a_n \{1+ \frac12 c_n^2(Z_{n,j}^++Z_{n,j}^-) + o(c_n^2)\} =  2a_n + b_n(Z_{n,j}^++Z_{n,j}^-) + o(b_n) \; .
  \end{align*}
  This proves that $(\|\mbx_{n,j}^+ - \mbx_{n,j}^-\|_q -2a_n)/b_n\convweak
  \Gamma_{j,1}^++\Gamma_{j,1}^-$. This yields~(\ref{eq:diamq-spherique-q>2}).
\end{proof}

\subsection{Case $1 \leq q < 2$}
Let $\Rset^d$ be split into $2^d$ isometric regions $Q_j$, $\pm j=1,\dots,2^{d-1}$ around each
``diagonal'' line $x_1=\pm x_2=\cdots=\pm x_d$, numbered in such a way that $Q_j=-Q_{-j}$ and that
$Q_1$ is the region which contains the point $\mathbf1=(1,\dots,1)$. For $q\in[1,2)$, a spherical
vector with a large $l^q$ norm must be close to one of the diagonals.

Define, $\psi_q(x) = d^{1/q-1/2} \psiT(d^{1/2-1/q}x)$ and $\phi_q(x) = \sqrt{\psi_q(x)/x}$.
\begin{theorem}
  \label{theo:lq-spherique-q<2} 
  Let $\mbx$ be as in Theorem~\ref{theo:norm-lq-q>2}.  If $1 \leq q < 2$, then
  \begin{align}
     \label{eq:normq-spherique-q<2}
     \pr(\|\mbx\|_q>x) \sim \frac{2^{3(d-1)/2} \Gamma(d/2)} {\Gamma(1/2) (2-q)^{(d-1)/2}} \;
     \phi_q^{d-1} (x) \pr \left(T>xd^{\frac12-\frac1q} \right) \; .
  \end{align}
  Moreover, conditionally on $\|\mbx\|_q>x$ and $\mbx\in Q_1$, as $x\to\infty$,
  \begin{align}
    \label{eq:convergence-conditionnelle-q<2}
    \left( \frac{\|\mbx\|_q-x}{\psi_q(x)} , \frac{\mbw-d^{-1/2}\mathbf1}{\phi_q(x)} \right)
    \convweak (E,\mbg) \; ,
  \end{align}
  where $E$ is an exponential random variable with mean 1 and $\mbg$ is a Gaussian vector
  independent of $E$ with covariance matrix
  \begin{align*}
    \Sigma = \frac1{d(2-q)} \begin{pmatrix}
      d-1 & -1 & \dots & -1 \\
      -1 & d-1 & \dots & -1 \\
      \vdots & & & \vdots \\
      -1 & \dots & -1 & d-1
    \end{pmatrix} \;
  \end{align*}
\end{theorem}
\paragraph{Comments}
\begin{itemize}
\item The form of the covariance matrix implies that the components of the vectors $\mbg$ sum up to
  zero. This is natural since $\mbg$ must be in the space tangent to the sphere at the point
  $d^{-1/2}\mathbf1$.
\item Here again, if $T^2$ has a $\chi^2$ distribution with $d$ degrees of freedom, the
  equivalent~(\ref{eq:normq-spherique-q<2}) is a particular case of
  \cite[Theorem~1 and Example~1]{hashorva:korshunov:piterbarg:2013}:
\begin{align*}
  \pr(\|\mbx\|_q^q>x) \sim \frac{2^d d^{\frac1q-\frac12}}{\sqrt{2\pi}(2-q)^{(d-1)/2}} x^{-1/q} \rme^{-\frac12 d^{1-2/q}x^{2/q}} \; .
\end{align*}
\end{itemize}

\begin{proof}[Proof of Theorem~\ref{theo:lq-spherique-q<2}]
  Let $P$ be an orthogonal matrix such that
  $d^{-1/2} P \mathbf1 = (1,0,\dots,0)'$ and define $g=f\circ P^{-1}$.  Note that $P\mbw$ is
  uniformly distributed on $\sphere{d-1}$, i.e.~has the same distribution as $\mbw$. For $f$
  continuous and compactly supported on $\Rset^d$, we have
\begin{align*}
  \esp \Big[ f  \Big(& \tfrac{\mbw-d^{-1/2}\mathbf1} {\phi_q(x)} \Big) \1{\{T >
    \frac{x+\psi_q(x)z} {\|\mbw\|_q}\}} \1{\{\mbw\in Q_1\}} \Big] \\
  & = \esp \left[ g \left(\tfrac{\mbw-(1,0,\dots,0)\}}{\phi_q(x)}\right) \1{\{T >
      \frac{x+\psi_q(x)z} {\|P^{-1}\mbw\|_q}\}}  \1{\{\mbw\in PQ_1\}} \right] \\
  & = \frac1{\beta_d} \int_{0}^1 \esp \left[ g \left(\tfrac{u-1}{\phi_q(x)}, \tfrac{1-u ^2}{\phi_q(x)}
      \tilde\mbw\right) \1{\{T > \frac{x+\psi_q(x)z} {\|P^{-1}(u,\sqrt{1-u^2}\tilde\mbw)\|_q}\}}
  \right] (1-u^2)^{(d-3)/2} \, \rmd u \\
  & = \frac1{\beta_d} \phi_q^{d-1}(x) \int_{0}^{1/\phi_q(x)} \esp \left[ g \left(-\phi_q(x) v,
      \sqrt{2v-\phi_q^2(x)v^2}      \tilde\mbw\right) \right. \\
  & \hspace*{2cm} \left.\1{\{T > \frac{x+\psi_q(x)z} {\|P^{-1}(1-\phi_q^2(x)v,\phi_q(x)
        \sqrt{2v-\phi_q^2(x)v^2}\tilde\mbw)\|_q}\}} \right] (2v-\phi_q^2(x)v^2)^{(d-3)/2} \, \rmd v  \; .
\end{align*}

Denote $\tilde \mbu = P^{-1}(0,\tilde\mbw)$. Then $\tilde\mbu \in\sphere{d-1}$ and moreover, 
\begin{align*}
  \langle \tilde\mbu, \mathbf1 \rangle = \langle P^{-1}(0,\tilde\mbw),\mathbf1 \rangle = \langle
  (0,\mbw) , P\mathbf1 \rangle = \langle
  (0,\mbw) , d^{1/2}(1,0,\dots,0) \rangle = 0 \; .
\end{align*}
In view of this, a second order Taylor expansion yields
\begin{align*}
  \| P^{-1} (1&-\phi_q^2(x)v,\phi_q(x) \sqrt{2v-\phi_q^2(x)v^2}\tilde\mbw)\|_q \\
  & = d^{\frac1q-\frac12} \left\{ 1 - d^{-1/2} \phi_q(x) \sqrt{2v} \sum_{i=1}^d \tilde\mbu_i +
    \phi_q^2(x) v\left((q-1) \|\tilde\mbu\|^2-1\right) + o_P(\phi_q^2(x)) \right\}  \\
  & = d^{\frac1q-\frac12} \left\{ 1 - \phi_q^2(x) (2-q) v + o_P(\phi_q^2(x)) \right\} \; .
\end{align*}
This yields, for $f$ continuous and compactly supported, 
\begin{align*}  
  \lim_{x\to\infty} \frac {\phi_q^{1-d}(x)} {\pr(T>xd^{\frac12-\frac1q})} & \esp \Big[ f \Big(
  \tfrac{\mbw-d^{-1/2}\mathbf1} {\phi_q(x)} \Big) \1{\{T > \frac{x+\psi_q(x)z} {\|\mbw\|_q}\}} \1{\{\mbw\in Q_1\}}  \Big]  \\
  & = \frac1{\beta_d} \, \rme^{-z} \int_{0}^\infty \esp[g(0,\sqrt{2v}\tilde\mbw)] (2v)^{(d-3)/2}
  \rme^{-(2-q)v} \, \rmd v \\
  & = \frac1{2(2-q)^{(d-1)/2}\beta_d} \, \rme^{-z} \int_{0}^\infty
  \esp[g(0,\sqrt{\tfrac{w}{2-q}}\tilde\mbw)] w^{(d-3)/2} \rme^{-w/2} \, \rmd w    \\
  & = \frac{2^{(d-3)/2} \Gamma((d-1)/2)}{(2-q)^{(d-1)/2}\beta_d} \, \rme^{-z}
  \esp[g(0,(2-q)^{-1/2} R\tilde\mbw)]    \\
  & = \frac{2^{(d-3)/2} \Gamma((d-1)/2)}{(2-q)^{(d-1)/2}\beta_d} \, \rme^{-z}
  \esp[f((2-q)^{-1/2} P^{-1}(0,R\tilde\mbw)]  \\
  & = \frac{2^{(d-3)/2} \Gamma((d-1)/2)}{(2-q)^{(d-1)/2}\beta_d} \, \rme^{-z}
  \esp[f((2-q)^{-1/2} R\tilde\mbu)] \; ,
\end{align*}
where $R^2$ has a $\chi^2$ distribution with ${d-1}$ degrees of freedom and is independent of
$\tilde\mbw$. Thus $R\tilde\mbw$ is a $(d-1)$ dimensional standard Gaussian vector. This implies
that $(2-q)^{-1/2}R\tilde\mbu$ is a $d$ dimensional Gaussian vector with covariance matrix
\begin{align*}
  \frac1{2-q} P^{-1} \begin{pmatrix} 0 & & \cdots & 0 \\ 0 & 1 & \cdots & 0 \\ \vdots & & \ddots &
    \vdots \\ 0 & \cdots & 0 & 1 \end{pmatrix} P = \frac1{d(2-q)} \begin{pmatrix}
    d-1 & -1 & \dots & -1 \\
    -1 & d-1 & \dots & -1 \\
    \vdots & & & \vdots \\
    -1 & \dots & -1 & d-1
  \end{pmatrix} = \Sigma \; .
\end{align*}
This also implies that the components of $R\tilde\mbu$ sum up to zero.  Summarizing, we have proved
that, for $f$ continuous and compactly supported
\begin{multline*}  
  \lim_{x\to\infty} \frac {\phi_q^{1-d}(x)} {\pr(T>xd^{\frac12-\frac1q})} \esp \Big[ f \Big(
  \tfrac{\mbw-d^{-1/2}\mathbf1} {\phi_q(x)} \Big) \1{\{T > \frac{x+\psi_q(x)z} {\|\mbw\|_q}\}}
  \1{\{\mbw\in Q_1\}} \Big]  \\
  = \frac{2^{(d-3)/2} \Gamma(d/2) }{\Gamma(1/2)(2-q)^{(d-1)/2} }  \, \rme^{-z}
  \esp[f(\mbg)] \; ,
\end{multline*}
where $\mbg$ is a Gaussian vector with mean zero and covariance matrix $\Sigma$.  Again, the
extension of the convergence to bounded continuous functions is done as in the proof of
Theorem~\ref{theo:norm-ddim}, using the bound~(\ref{eq:borne-uniforme-avec-z}). This
proves~(\ref{eq:convergence-conditionnelle-q<2}). Summing this equivalence over the $2^d$ regions
$Q_j$ yields~(\ref{eq:normq-spherique-q<2}).
\end{proof}

Let $\mbu$ be as in~(\ref{eq:def-mbu}).  Theorem~\ref{theo:lq-spherique-q<2} yields that conditionally
on $\|\mbx\|_q>x$ and $\mbu\in Q_1$,
\begin{align}
  \label{eq:exces-condi-q<2}
  \left( \frac{\|\mbx\|_q-x}{\psi_q(x)} \; , \; \frac{\mbu - d^{-1/q} \mathbf1}{\phi_q(x)} \right)
  \convweak (E, d^{1/2-1/q} \mbg) \; .
\end{align}

Theorem~\ref{theo:lq-spherique-q<2} and the convergence~(\ref{eq:exces-condi-q<2}) can be adapted to
each region $Q_j$. For $j=1,\dots,2^d$, let $\beps_j$ be the point of
$\{-1,1\}^d\setminus\{\mathbf1\}$ which is in $Q_j$. Then, conditionally on $\|\mbx\|_q>x$ and
$\mbu\in Q_j$,
\begin{align*}
  \left( \frac{\|\mbx\|_q-x}{\psi_q(x)} \; , \; \frac{\mbu - d^{-1/q} \beps_j}{\phi_q(x)} \right)
  \convweak (E, d^{1/2-1/q} \mbg_j) \; ,
\end{align*}
where $\mbg_j = (\varepsilon_1G_1,\dots,\varepsilon_dG_d)$ and $(G_1,\dots,G_d)$ is a Gaussian
vector with zero mean and covariance matrix $\Sigma$.

The previous results can be translated into point process convergence. Let $a_n$ be the $1-1/n$
quantile of the distribution of $\|\mbx\|_q$. Define $b_n = \psi_q(a_n)$ and $c_n = \sqrt{b_n/a_n}$.
For $j=1,\dots,2^d$ and $i=1,\dots,n$, define
\begin{align*}
  P_{n,i,j} = \left(\tfrac{\|\mbx_i\|_q-a_n}{b_n},\tfrac{\mbu_{i}-d^{-1/q}\beps_j}{c_n}\right) \; .
\end{align*}

\begin{corollary}  
  \label{coro:ppconv-q<2}
  Let $\{\mbx_i,i\geq1\}$ be a sequence of i.i.d.~random vectors with the same distibution as $\mbx$
  which satisfies the assumptions of Theorem~\ref{theo:norm-lq-q>2}.  Then,
  \begin{align*}
    \sum_{i=1}^n \delta_{P_{n,i,j}} \1{\{\mbu_i\in Q_j\}} \convweak \sum_{i=1}^\infty
    \delta_{P_{i,j}}
  \end{align*}
  where $P_{i,j} = (\Gamma_{i,j}, d^{1/2-1/q} \mbg_i)$, $\sum_{i=1}^\infty \delta_{\Gamma_{i,j}}$
  are independent Poisson point processes with mean measure $2^{-d}\rme^{-x} \rmd x$ and
  $\{\mbg_{i,j},i\geq1\}$, $j=1,\dots,2^d$ are independent sequences of i.i.d.~Gaussian vectors with
  the same distribution as $\mbg_j$, independent of $\{\Gamma_{i,j},i\geq1\}$, $j=1\dots,2^d$.
\end{corollary}

These point process convergences yield the asymptotic behavior of the diameter.

\begin{theorem}
  \label{theo:diametre-lq-spherique-q<2}
   Let $\{\mbx_i,i\geq1\}$ be a sequence of i.i.d.~random vectors with the same distibution as
    $\mbx$ which satisfies the assumptions of Theorem~\ref{theo:norm-lq-q>2}.  If $1 \leq q < 2$, then, 
  \begin{align}
     \label{eq:diamq-spherique-q<2}
     \frac{\diamq(\mbx) - 2a_n}{b_n} \convweak \; \max_{1 \leq j \leq 2^{d-1}} \max_{i,i'\geq1}
     \left\{\Gamma_{i,j}^++\Gamma_{i',j}^- - \frac{q-1}4 \sum_{\ell=1}^d (G_{i,j,\ell}^+ +
       G_{i',j,\ell}^-)^2 \right\}  \; ,
   \end{align}
   where $\Gamma_{i,j}^\pm$, $i\geq1$ $j=1,\dots,2^{d-1}$ are the points of independent Poisson
   point processes on $(-\infty,\infty]$ with mean measure $2^{-d} \rme^{-x} \, \rmd x$ and
   $\mbg_{i,j}^\pm=(G_{i,j,1}^\pm,\dots,G_{i,j,d}^\pm)$, $i\geq1$, $j=1,\dots,2^{d-1}$ are
   i.i.d.~Gaussian vectors with covariance matrix $\Sigma$
\end{theorem}

\paragraph{Comments} For $q=1$, the corrective terms in~(\ref{eq:diamq-spherique-q<2}) vanish and so
the limiting distribution of the diameter is $\max_{j=1,\dots,2^{d-1}}
\Gamma_{1,j}^++\Gamma_{1,j}^-$. If $d>2$, it differs from the case $q>2$ since the space is split
into more regions (there are $2^{d-1}$ diagonals and $d$ axes).

\begin{proof}[Proof of Theorem~\ref{theo:diametre-lq-spherique-q<2}] 
  The diameter will be achieved by points nearly diametrically opposed and close to one of the
  diagonals. More precisely,
  \begin{align*}
    \lim_{n\to\infty} \pr \left( \diamq(\mbx) = \max_{1\leq j \leq 2^{d-1}} \max\{\|\mbx_i -
      \mbx_{i'}\|_q \mid \mbx_i\in Q_j,\mbx_{i'} \in Q_{-j}, 1 \leq i,i'\leq n\} \right) = 1 \; .
  \end{align*}
  In order to obtain the convergence of each sub-maximum, we proceed as in the proof of
  Theorem~\ref{theo:diam-k=1}. The main step is the following.  Define
  \begin{align*}
    r_{n,i} = a_n + b_n r_i \; , \ i=1,2 \; , \ w_{n,1} = d^{-1/2} \1{} - c_n u \; , w_{n,2} =
    -d^{-1/2} \1{} - c_n v \; ,
  \end{align*}
  where $u$ and $v$ are such that $\|u\|_q = \|v\|_q = 1$. This implies that
  \begin{align*}
    c_n \sum_{i=1}^d u_i & = - d^{\frac1q} \frac{q-1}2 c_n^2 \sum_{j=1}^d u_j^2  + o(c_n^2) \; , \\
    c_n \sum_{i=1}^d v_i & = d^{\frac1q} \frac{q-1}2 c_n^2 \sum_{j=1}^d v_j^2 + o(c_n^2) \; .
  \end{align*}
  This yields the expansion
  \begin{align*}
    \|r_{n,1}w_{n,1} - r_{n,2}w_{n,2}\|_q = 2a_n \left\{ 1 + \frac12c_n^2(r_1+r_2) -
      \frac{d^{\frac2q-1}(q-1)}8 \sum_{j=1}^d (u_j+v_j)^2 + o(c_n^2) \right\}\; .
  \end{align*}
  This implies the convergence
  \begin{align*}
    \lim_{n\to\infty} \frac{\|r_{n,1}w_{n,1} - r_{n,2}w_{n,2}\|_q -2a_n}{b_n} = r_1 + r_2 - \frac{
      d^{\frac2q-1} (q-1)}4 \sum_{j=1}^d (u_j+v_j)^2 \; .
  \end{align*}
  The rest of the proof is exactly along the lines of the proof of Theorem~\ref{theo:diam-k=1}.
\end{proof}

\section{Further generalizations}
\label{sec:bidim}

There are many ways to generalize the results of the previous sections, and because of the very
local nature of the behavior of random vectors in the domain of attraction of the Gumbel
distribution, it is possible to build all kind of ad hoc examples to illustrate nearly any type of
behaviors. In this section we will only briefly describe several reasonable generalizations of
elliptical distributions.

One possibility is to consider a random vector $\mbx$ that has the representation $\mbx = T \mbw$,
where $\mbw$ is a random vector on the sphere $\sphere{d-1}$, no longer assumed to be uniformly
distributed, and $T$ is a positive random variable, independent of $\mbw$.  A second possibility is
to assume that the vector $\mbx$ can be expressed as $\mbx = Tg(\mbw)$, where $\mbw$ is uniformly
distributed on $\sphere{d-1}$ and $g$ is a bounded continuous function.  This model includes the
previous one if the function $g$ takes values in the unit sphere. These models were used by
\cite{MR2739357} and \cite{barbe:seifert:2013} in the investigation of conditional limit laws of a
bivariate vector given that one component is extreme. In such a model, the behavior of the vector
given that its norm is large and the behavior of the diameter will be determined by the maxima of
the function $\|g\|$. If they are isolated points, the localization phenomenon will arise and
results such as Theorem~\ref{coro:norm-ddim} and~\ref{theo:diam-k=1} may be obtained. Otherwise, if
$g$ is constant on non empty open subsets of the sphere, we rather expect to obtain results similar
to Theorem~\ref{theo:diam-k>1}.

Another way to generalize the elliptical distributions is to consider vectors whose distribution has
a density on $\Rset^d$ of the form $f(x) = \rme^{-U(x)}$ where $U$ is a continuous function on
$\Rset^d$ and the level sets of $U$ are closed and convex and $U$ satisfies some type of
multivariate regular variation or asymptotic homogeneity. This type of assumptions has been used in
\cite{balkema:embrechts:2007} to obtain conditional limit laws of a vector given that one component
is extreme and by \cite{hsing:rootzen:2005} in the study of the longest edge of the minimum spanning
tree of a random sample.

We leave this last direction as the subject of future research. In the following subsections, we
give without proof several bidimensional examples. We only consider the Euclidean norm.

\subsection{Generalized spherical distributions}
\label{sec:continous-extension}

\newcommand\exponentcos{\gamma}
\newcommand\constantcos{C_0}

Assume that $\mbx=T(\cos\Theta,\sin\Theta)$ where $T$ and $\Theta$ are independent and the support
of the distribution of $\Theta$ is $[0,\theta_0]$, $\theta_0\in(0,2\pi]$. In this case, it holds
that $\|\mbx\|=T$ and as previously, we denote the quantile of order $1-1/n$ of $\|\mbx\|$ by $a_n$
and define $b_n=\psiT(a_n)$.

The main question in this case is the existence of nearly diametrically opposed vectors in the
sample cloud. If $\theta_0<\pi$, then there will be none, and therefore the diameter cannot behave
like twice the norm.

The case $0<\theta_0\leq\pi/3$ is trivial since $\|\mbx_1 -\mbx_2 \| \leq \|\mbx_1\|\vee\|\mbx_2\|$
if the angle between $\mbx_1$ and $\mbx_2$ is less than $\pi/3$. In concrete terms, the distance
between two points whose angle is less than $\pi/3$ is always smaller than their norms. This implies
that $M_n^{(2)}(\mbx) \leq M_n(\mbx)$. Define $m_n(\mbx) = \min_{1 \leq i \leq n} \|\mbx_i\|$ and
let $\hat{\mbx}_n$ and $\check{\mbx}_n$ be points in the sample such that $\|\hat{\mbx}_n\|=M_n$ and
$\|\check {X}_n\|=m_n$. Then, by the triangle inequality
\begin{align*}
  M_n^{(2)}(\mbx) \geq d(\hat{\mbx}_n,\check{\mbx}_n) \geq M_n(\mbx) - m_n(\mbx) \; .
\end{align*}
Therefore we conclude that $(M_n(\mbx)-\diam2(\mbx))/m_n(\mbx) \to_P1$ and
\begin{align}
  \lim_{n\to\infty} \pr(\diam2(\mbx) \leq a_n + b_nx) = \rme^{-\rme^{-x}} \; .
 \end{align}

 If $\theta_0\in(\pi/3,\pi)$, then there will be no vectors nearly diametrically opposed, but this
 case will differ from the case $\theta_0\in[\pi,2\pi]$ only by constants. As can be seen from the
 proof of Theorem~\ref{theo:diam-k>1} and \cite[Theorem~1.1]{jammalamadaka:janson:2012}, if
 $\theta_0\geq\pi$, the asymptotic distribution of the diameter is determined by the behavior of
 $\cos(\Theta_1-\Theta_2)$ at -1. If $\theta_0\in(\pi/3,\pi)$, then it is determined by the behavior
 of $\cos(\Theta_1-\Theta_2)$ when the angle between $\Theta_1$ and $\Theta_2$ is the largest, here
 $\theta_0$. Apart from this difference, the proof of \cite[Theorem~1.1]{jammalamadaka:janson:2012}
 can be copied line by line to obtain the following result.

\begin{proposition}
  \label{theo:maxdn-cone}
  Let $\{\mbx_i,i\geq1\}$ be a sequence i.i.d.~random vectors whose distribution can be expressed as
  $T(\cos\Theta,\sin\Theta)$, where $T$ and $\Theta$ are independent, $T$ satisfies
  Assumption~\ref{hypo:strengthening-psi}, $\Theta$ has support in $[0,\theta_0]$,
  $\theta_0\in(\pi/3,\pi)$ and
  \begin{align*}
    \pr(\cos(\Theta_1-\Theta_2) - \cos(\theta_0\wedge\pi) < \epsilon) = \constantcos
    \,\epsilon^\exponentcos + o\big(\epsilon^\exponentcos\big) \; ,
  \end{align*}
  where $\Theta_1,\Theta_2$ are i.i.d.~with the same distribution as
  $\Theta$, $C_0>0$ and $\gamma\geq0$. Then
  \begin{align}
    \label{eq:maxdn-cone}
    \lim_{n\to\infty} \pr \left( \frac{\diam2-\kappa_0a_n}{2b_n/\kappa_0} +
      \exponentcos\log\frac{a_n}{b_n} - \log\log\frac{a_n}{b_n} - \log C_{\exponentcos,\kappa_0}
      \leq x \right) = \rme^{-\rme^{-x}} \; ,
\end{align}
with
\begin{align*}
  \kappa_0 & = \sqrt{2(1-\cos(\theta_0\wedge\pi))} \in (1,2] \; , \ \ C_{\exponentcos,\theta_0} =
  \constantcos \kappa_0^{-1} 2^{\exponentcos} \exponentcos\Gamma(\exponentcos+1) \; .
\end{align*}

\end{proposition}

Let us give an example.  Assume that the distribution of $\Theta$ has a density $f_\Theta$ on
$[0,\pi]$ defined by $f_ \Theta(x)= ({6}/{\pi^3}) x(\pi-x) \1{[0,\pi]}(x)$. We obtain
\begin{align*}
  \pr(1+\cos(\Theta_1-\Theta_2) < \epsilon) = \frac{12}{\pi^4}\,\epsilon^2 + o(\epsilon^2) \; .
\end{align*}
Thus~(\ref{eq:maxdn-cone}) holds with $\kappa_0=2$, $\constantcos=12\pi^{-4}$ and $\exponentcos=2$.

\subsection{Generalized elliptical distributions}
\label{sec:discrete-extension}

Let $u$, $v$ be two continuous functions defined on $[0,1]$ such that $u(0)=u(1)$ and $v(0)=v(1)$
and such that the curve $\gamma(s) = (u(s),v(s))$ is simple. Define a bivariate random vector $\mbx$
by
\begin{align*}
  \mbx = T(u(S),v(S)) \; ,
\end{align*}
where $T$ and $S$ are independent and $S$ is uniformly on $[0,1]$. We call such a vector a
generalized elliptical vector since elliptical vectors are obtained by choosing $u(s) = \cos(2\pi
s)$ and $v(s) = \cos(2\pi s-U_0)$.

Define $\ell(s) = \sqrt{u^2(s)+v^2(s)}$ and assume that $s$ has exactly $q$ maxima
$0<s_1,\dots,s_q<1$ which are isolated points, i.e.~$\sup_{s\in[0,1]} \ell(s)=\max_{i=1,\dots,q}
\ell(s_i)$ and for each $i=1,\dots,q$, there exists $\epsilon>0$ such that $\ell(s)<\ell(s_i)$ for
all $s\in(s_i-\epsilon,s_i+\epsilon)$, $s\ne s_i$. Assume moreover that $\ell$ is twice
differentiable, and that $\ell''(s_i)<0$ for $i=1,\dots,q$.  Let $0=t_0 < t_1 < \cdots < t_q = 1$
define a partition of $[0,1]$ such that $s_i \in (t_{i-1},t_i)$, $i=1,\dots,q$.

Define $\phi(x) = \sqrt{\psiT(x)/x}$ and for $i=1,\dots q$, $m_i = \ell(s_i)$ and $\tau_i^2 =
-m_i/\ell''(s_i)$. Adapting the proof of Theorem~\ref{theo:norm-ddim}, we obtain
\begin{align*}
  \lim_{x\to\infty} \frac{ \pr(\|\mbx\|/m_i > x + \psiT(x) z \; , \; S \leq s_i + \phi(x) u , \; S
    \in (t_{i-1},t_i)) }{\sqrt{2\pi\tau_i^2} \; \phi(x)\pr(T>x)} = \rme^{-z} \Phi_{\tau_i}(u) \; .
\end{align*}
The large observations are localized around the directions of the points $\gamma(s_i)$,
$i=1,\dots,q$.  Define $m=\max_{i=1,\dots,q} m_i$ and $\tau = \sum_{i:m_i=m} \tau_i$. Noting that
$\pr(T>x/m_i) = o(\pr(T>x/m))$ if $m_i<m$, the previous expansion yields
\begin{align*}
  & \pr(\|\mbx\| > x) \sim \sqrt{2\pi\tau^2} \sqrt{\frac{\psiT(x/m)}{x/m}} \pr(T>x/m) \; .
\end{align*}
This implies that an auxiliary function for $\|\mbx\|$ is $m\psiT(x/m)$.  This idea has been
exhaustively investigated in higher dimension under the assumption that $T$ has a $\chi^2$
distribution by \cite[Theorem~1 and~2]{hashorva:korshunov:piterbarg:2013}.

We expect the diameter of the cloud to be achieved by pairs of points with large norms and which are
nearly in the directions of the points $\gamma(s_i)$ and $\gamma(s_j)$ with maximum distance. We
have obtained the limiting distribution of the diameter only when the two points with maximum
distance are diametrically opposed.

Assume that $\gamma(s_1)$ and $\gamma(s_2)$ are diametrically opposed and that
\begin{align*}
  \|\gamma(s_1)-\gamma(s_2)\|  = \max_{1 \leq i < j \leq q}   \|\gamma(s_i)-\gamma(s_j)\| \; .
\end{align*}
Assume for simplicity that this maximum is achieved only once. Let $a_n$ be the $1-1/n$ quantile of
$\|\mbx\|/m$ and $b_n=\psiT(a_n)$.  Adapting the proof of Theorem~\ref{theo:diam-k=1}, we obtain 
\begin{align*}
  \frac{\diam2(\mbx) - (m_1+m_2) a_n} {b_n} \convweak \max_{i,j\geq1} \left\{ m_1 \Gamma_i^+ + m_2
    \Gamma_j^- - \tfrac{m_1m_2}{2(m_1+m_2)}
    \left(\tfrac{v'(s_1)}{\ell''(s_1)}G_i^+-\tfrac{v'(s_2)}{\ell''(s_2)}G_j^- \right)^2 \right\} \;  ,
\end{align*}
where $\{\Gamma_i^+,i\geq1\}$ and $\{\Gamma_i^-,i\geq1\}$ are the points of two independent Poisson
point processes with mean measure $\frac1q \rme^{-x}\,\rmd x$, independent of the
i.i.d.~standard Gaussian random variables $G_{i}^+,G_{j}^-$, $i,j\geq1$.

The problem when the points $\gamma(s_i)$, $\gamma(s_j)$ which achieve the maximum distance are not
diametrically opposed is that the rate at which the vector with large norms concentrate to the
directions of the points $\gamma(s_i)$ and $\gamma(s_j)$ is not fast enough to apply the arguments
of the proof of Theorem~\ref{theo:diam-k=1}. We leave this problem and higher dimensional extensions
to future research.

\subsection{Different rates of localization}
\label{sec:different}
The rate of localization of the vectors around the direction where the norm can be large is
$\sqrt{a_n/b_n}$ in the previous examples. This is due to the regularity of the curve
$\gamma$. Different rates may be obtained if the norm is not twice differentiable at its maxima but
has some regular variation property. Consequently, different limiting distributions are also
obtained. We give one example. 

Let $U$ be uniformly distributed on $[0,2\pi]$, $q\in(1/2,1)$, $a>1$ and $T$ independent of
$U$. Define 
\begin{align*}
  \mbx = T (a \cos(|U|^q), \sin(|U|^{q})\sgn(U)) \ .
\end{align*}
The maximum of the function $a^2\cos^2(|\theta|^q)+ \sin^2(|\theta|^q)$ is achieved when $\theta=0$
or~$\theta=\pi$.  

Define $\psi_a(x) = a\psiT(x/a)$ and $\phi_{a,q}(x) = \{\psi_a(x)/x\}^{1/(2q)}$. Let $Z_q$ be a
random variable whose distribution admits the density $q 2^{-1/(2q)} \Gamma^{-1}(1/(2q))
\rme^{-\frac12|x|^{2q}}$ with respect to Lebesgue's measure on $\Rset$ and let $E$ be an exponential
random variable with mean~1. Then, conditionally on $\|\mbx\|>x$ and $\cos U>0$,
\begin{align*}
  \left(\frac{\|\mbx\|-x}{\psi_a(x)},\frac{U} {\phi_{a,q}(x)}\right)
  \convweak \left(E,\left(\tfrac{a^2}{a^2-1}\right)^{\frac1{2q}} Z_q \right)
\end{align*}
where $E$ and $Z_{q}$ and independent.  A similar convergence holds conditionally on $\cos
U<0$. This implies that $\psi_a$ is an auxiliary function of $\|\mbx\|$. This yields an analogue of
Theorem~\ref{theo:diam-k=1} where the distribution $Z_q$ plays the role of the standard Gaussian
distribution. Let $a_n$ be the $1-1/n$ quantile of $\|\mbx\|$ and let $b_n = \psi_a(a_n)$. Then,
\begin{align*}
  \frac{\diam2-2a_n}{b_n} \convweak \max_{i,j\geq1} \; \left\{ \Gamma_i^+ + \Gamma_j^- -
    \frac{a^{1-1/q}}{4(a^2-1)^{1/(2q)}} (Z_{i}^+-Z_{j}^-)^2 \right\} \; ,
\end{align*}
where $\{\Gamma_i^\pm$, $i\geq1\}$ are the points of two independent Poisson point
processes with mean measure $\frac12\rme^{-x}\, \rmd x$ and $Z_i^\pm$ are
i.i.d.~random variables with the same distribution as $Z$, and independent of
the point processes.

\section{Proof of Lemmas~\ref{lem:conditioning-uniform-sphere} 
and~\ref{lem:poisson-convergence}}
\label{sec:proof-lemmas}

\begin{proof}[Proof of Lemma~\ref{lem:conditioning-uniform-sphere}]
  It is known that $\mbw$ is uniformly distributed on $\sphere{d-1}$ if and only if $\mbw =
  \|\mbx\|^{-1}\mbx$ where $\mbx$ is a $d$-dimensional standard Gaussian vector.  Equivalently,
  $\mbw$ is uniformly distributed on $\sphere{d-1}$ if and only if $R\mbw$ is a $d$-dimensional
  standard Gaussian vector, where $R^2$ has a $\chi^2$ distribution with $d$ degrees of freedom and
  is independent of $\mbw$. Let $R$ be such a random variable and define $\mbx = R\mbw$. The
  coordinates $X_1,\dots,X_d$ of $\mbx$ are i.i.d.~standard Gaussian random variables. It is then
  easily seen that
  \begin{align*}
    \unifk = \frac{(X_{1},\dots,X_k)}{\|(X_{1},\dots,X_k)\|} \; ,
  \end{align*}
  hence $\unifk$ is uniformly distributed on $\sphere{k}$. Moreover, $R_k = \|(X_{1},\dots,X_k)\|$
  is independent of $\tilde\mbw_k$. Noting that
  \begin{align}
     \label{eq:representation-uniforme-spherique}
    (W_{k+1},\dots,W_d)  = \frac{(X_{k+1},\dots,X_d)}{\sqrt{R_k^2+X_{k+1}^2 + \cdots + X_d^2} }
  \end{align}
  and that $\unifk$ is independent of $X_{k+1},\dots,X_d$, we obtain the independence of $\unifk$
  and $(W_{k+1},\dots,W_d)$.

  Let $f$ be compactly supported on $\Rset^d$ and $g$ be the density of $(W_{k+1},\dots,W_d)$.
  Since $\unifk$ is independent of $(W_{k+1},\dots,W_d)$, it holds that
  \begin{align*}
    s^{k} \, \esp & [ f(W_1,\dots,W_k,sW_{k+1},\dots,sW_d) ] \\
    & = s^{k} \, \esp \left[ f \left( \sqrt{1 - W_{k+1}^2-\cdots-W_d^2} \, \unifk,sW_{k+1},\dots,sW_d \right) \right]  \\
    & = s^k \int_{[-1,1]^{d-k}} \esp \left[ f \left(\sqrt{1-u_{k+1}^2-\cdots-u_d^2} \,
        \unifk,su_{k+1}, \dots,su_d \right) \right] g(u_{k+1},\dots,u_d) \rmd u_{k+1} \dots \rmd u_d
    \\
    & = \int_{[-s,s]^{d-k}} \esp \left[ f \left( \sqrt{1-s^{-2}(u_{k+1}^2+\cdots+u_d^2)}
        \unifk,u_{k+1},\dots,u_d \right) \right]
      g\left(\tfrac{u_{k+1}}s,\dots,\tfrac{u_d}s\right) \rmd u_{k+1} \dots \rmd u_d \\
      & \to g(0) \int_{\Rset^k} \esp[f(\unifk,u_{k+1},\dots,u_d) ] \, \rmd u_{k+1} \dots \rmd u_d \;
      .
  \end{align*}
  Let us now compute $g(0)$. Using the representation~(\ref{eq:representation-uniforme-spherique}),
  we have, for any bounded measurable function $f$ on $\Rset^{d-k}$,
  \begin{align*}
    \esp & [f(W_{k+1},\dots,W_d)] \\
    & = \int_0^\infty \int_{\Rset^{d-k}} f \left( \tfrac{(u_{k+1},\dots,u_d)}
      {\sqrt{r+u_{k+1}^2+\cdots+u_d^2}} \right) \rme^{-\frac12(u_{k+1}^2+\cdots+u_d^2)}
    r^{\frac{k}2-1} \rme^{-r/2} \frac{\rmd r \rmd u_{k+1} \dots \rmd u_d} {2^{k/2}
      \Gamma(k/2) \, (2\pi)^{(d-k)/2}}    \\
    & = \int_{[-1,1]^{d-k}} f(w_{k+1},\dots,w_d) g(w_{k+1},\dots,w_d) \rmd w_{k+1} \dots \rmd w_d \;
    ,
  \end{align*}
  with
  \begin{align*}
    g(w_{k+1},\dots,w_d) & = \frac{1}{2^{d/2}\pi^{(d-k)/2}\Gamma(k/2)} \int_0^\infty
    J(r,w_{k+1},\dots,w_d) r^{\frac{k}2-1} \rme^{-r/2}  \rmd r
  \end{align*}
  and $J(r,w_{k+1},\dots,w_d)$ is the Jacobian determinant of the change of variable 
  \begin{align*}
    (r,u_{k+1},\dots,u_d) \to (r,w_{k+1},\dots,w_d)=\left(r, \tfrac{(u_{k+1},\dots,u_d)}
      {\sqrt{r+u_{k+1}^2+\cdots+u_d^2}} \right) \; .
  \end{align*}
  It is readily checked that $J(r,0,\dots,0) = r^{(d-k)/2}$, hence
  \begin{align*}
    g(0) & = \frac{1} {2^{d/2}\pi^{(d-k)/2} \Gamma({k}/2)} \int_0^\infty r^{\frac{d}2-1} \rme^{-r/2}
    \rmd r = \frac{ \Gamma(\frac d2)}{\pi^{(d-k)/2} \Gamma({k}/2)} \; .
  \end{align*}
This yields the constant in~(\ref{eq:spherical-dilate}).
\end{proof}

\subsubsection*{Proof of Lemma~\ref{lem:poisson-convergence}}
We need several preliminary results.

For $i=1,2$, define $U_i^{(k)} = \sqrt{1 - \sum_{q=k+1}^d W_{i,q}^2}$. Then $(W_{i,1},\dots,W_{i,k})
= U_i^{(k)}\unifk_{i}$ where $\unifk_{i}$ is uniformly distributed on $\sphere{k-1}$ and
\begin{align*}
  \mby_i & = T_i(\sqrt{\lambda_1}U_i^{(k)}
  \unifk_{i},\sqrt{\lambda_{k+1}}W_{i,k+1},\dots,\sqrt{\lambda_d}
  W_{i,d})\; .
\end{align*}

Write
\begin{align*}
  \|\mby_1 - \mby_2\| & = \sqrt{\lambda_1} (T_1 + T_2) -
  \sqrt{\lambda_1} b_n^T A_n
\end{align*}
with $A_n =  h_n(1+\langle\unifk_{1}, \unifk_{2}\rangle
,W_{1,k+1},\dots,W_{1,d},W_{2,k+1},\dots,W_{2,d})$,
\begin{align*}
  h_n(s,u,v) & = \frac{T_1+T_2}{b_n^T}\{1-\sqrt{1 - (c_n^T)^2g_n(s,u,v)}\} \; , \\
  g_n(s,u,v) & = \frac{2a_n^TT_1T_2}{b_n^T(T_2+T_2)^2} \left\{ 1 - (1-s)
    \sqrt{1-\sum_{q=k+1}^d u_q^2} \sqrt{1-\sum_{q=k+1}^d v_q^2} +
    \sum_{q=k+1}^d \frac{\lambda_q}{\lambda_1} u_q v_q \right\}  \\
  & \ \ \ + \frac{T_1^2}{(T_1+T_2)^2} \sum_{q=k+1}^d
  \frac{\lambda_1-\lambda_q}{\lambda_1} u_q^2 +
  \frac{T_2^2}{(T_1+T_2)^2} \sum_{q=k+1}^d
  \frac{\lambda_1-\lambda_q}{\lambda_1} v_q^2 \; ,
\end{align*}
where $(u,v)=(u_{k+1},\dots,u_d,v_{k+1},\dots,v_d)\in\Rset^{2(d-k)}$
and $c_n^T = \sqrt{b_n^T/a_n^T}$. The following Lemma gives the limit
of the suitably rescaled functions $g_n$ and $h_n$. The proof is
elementary and is omitted.

\begin{lemma}
  \label{lem:convergence-hn}
  Let $\{\omega_n\}$ be a sequence of positive numbers such that $\omega_n =
  O(\log(a_n^T/b_n^T))$ and set $s_n=\log\omega_n$. Define the event
  $\mct_n = \{a_n^T - b_n^T(\omega_n+s_n) \leq T_1,T_2 \leq a_n^T +
  b_n^T s_n\}$.  Then, almost surely,
\begin{align*}
  \lim_{n\to\infty} h_n((c_n^T)^2s,c_n^Tu,c_n^Tv) \1{\mct_n} =
  \lim_{n\to\infty} g_n((c_n^T)^2s,c_n^Tu,c_n^Tv)\1{\mct_n} = g(s,u,v)
  \; ,
\end{align*}
locally uniformly, with
\begin{align}
  \label{eq:var-cor-uv}
  g(s,u,v) = \frac12 s + \frac{1}{2} \sum_{q=k+1}^d \frac{u_q^2 - 2 \rho_q u_q v_q +
    v_q^2}{\varpi_q^2(1-\rho_q^2)}\; , \rho_q = \frac{\lambda_q}{2\lambda_1-\lambda_q} \; , \ \
  \varpi_q^2 = \frac{2\lambda_1-\lambda_q}{2\lambda_1-2\lambda_q} \; .
\end{align}
Moreover, there exists a constant $c>0$ such that
\begin{align}
  \label{eq:borne-hn-integrable}
  h_n((c_n^T)^2s,c_n^Tu,c_n^Tv) \1{\mct_n} \geq c \left\{s+\sum_{q=k+1}^d (u_q^2+v_q^2) \right\}
\end{align}
\end{lemma}

\begin{lemma}
  \label{lem:strengthening-log}
  Under Assumption~\ref{hypo:strengthening-psi}, for any sequence
  $\{\omega_n\}$ such that $\omega_n = O(\log(a_n^T/b_n^T))$, for all
  $z\in\Rset$,
  \begin{align}
    \label{eq:strengthening}
    \lim_{n\to\infty} n \rme^{-\omega_n} \pr(T>a_n^T - b_n^T \omega_n
    + b_n^T z) = \rme^{-z} \; .
  \end{align}
\end{lemma}

\begin{proof}
  Denote $\tilde{a}_n = a_n^T - b_n^T \omega_n$ and $\tilde{b}_n =
  \psiT(\tilde{a}_n)$.  For any sequence $\{r_n\}$ that tends to
  infinity, the convergence $\pr(T>r_n+\psi(r_n)z)/\pr(T>r_n)$ is
  locally uniform with respect to~$z\in\Rset$. Under
  Assumption~\ref{hypo:strengthening-psi}, it holds that
  $\tilde{b}_n/b_n\to1$.  Thus,
  \begin{align*}
    \frac{\pr(T> \tilde{a}_n + b_n^Tz)}{\pr(T>\tilde{a}_n)} =
    \frac{\pr(T>\tilde{a}_n +
      \tilde{b}_n\frac{b_n^T}{\tilde{b}_n}z)}{\pr(T>\tilde{a}_n)} \to
    \rme^{-z} \; .
  \end{align*}
  Let us now prove that
  \begin{align}
    \label{eq:equiv-tilde}
    \lim_{n\to\infty} n \rme^{-\omega_n} \pr(T > \tilde{a}_n) = 1 \; .
  \end{align}

Using the representation of $\pr(T>x)$ in~(\ref{eq:vonmises}), we have
\begin{align*}
  \rme^{-\omega_n} \frac{\pr(T>\tilde{a}_n)}{\pr(T>a_n^T)} & =
  \frac{\vartheta(\tilde a_n)}{\vartheta(a_n^T)} \exp
  \int_0^{\omega_n} \left(\frac{b_n}{\psi(a_n-b_ns)} - 1 \right)\;
  \rmd s \; .
\end{align*}

Since $\omega_n = O(\log(a_n^T/b_n^T))$, the bound~(\ref{eq:strengthening-psi-square}) in Assumption~\ref{hypo:strengthening-psi} implies that
\begin{align*}
\int_0^{\omega_n} \left| \frac{b_n}{\psi(a_n-b_ns)} - 1 \right| \; \rmd s \leq \omega_n \sup_{|s|\leq \omega_n} \left| \frac{b_n}{\psi(a_n-b_ns)} - 1 \right| \to 0 \; .
\end{align*}
Since the function $\vartheta$ has a positive finite limit at infinity, this yields~(\ref{eq:equiv-tilde}).
\end{proof}

For any sequence $\{\omega_n\}$, define $s_n = \frac12 \log \omega_n$, the event $\mct_n=\{a_n^T-b_n^T(\omega_n+s_n) \leq T_1,T_2 \leq a_n^T+b_n^Ts_n\}$ and
for $z\in\Rset$,
\begin{align}
  \label{eq:dev-bar-kn} 
  K_n(z) = \frac{n^2 \rme^{-\omega_n}}{\omega_n} \pr(T_1+T_2 > 2a_n^T-b_n^T\omega_n+b_n^T z\; ;
  \mct_n) \; .
\end{align}

\begin{lemma}  
  \label{lem:convergence-JJ}
  If Assumption~\ref{hypo:strengthening-psi} holds, then, for any sequence $\{\omega_n\}$ such that
  $\omega_n\to\infty$ and $\omega_n = O(\log(a_n/b_n))$, and for all $z\in\Rset$,
\begin{align} 
  \label{eq:convergence-JJ12}
  \lim_{n\to\infty} K_n(z) = \rme^{-z} \; .
\end{align}
\end{lemma}

\begin{proof}
  The proof of the convergence~(\ref{eq:convergence-JJ12}) is
  a consequence of Lemmas~3.5 to 3.9
  in~\cite{jammalamadaka:janson:2012} under~(\ref{eq:strengthening})
  as an assumption.
\end{proof}

\begin{lemma}
  \label{lem:borne-uniforme-spherique}
  If Assumption~\ref{hypo:strengthening-psi} holds, then for each $p>0$, each sequence
  $\{\omega_n\}$ such that $\omega_n = O(\log(b_n^T/a_n^T)$, there exists a constant $C$ such that,
  for large enough $n$ and all $y \geq 0$,
\begin{align}
  \label{eq:borne-uniforme-spherique}
  \sup_{u\in(-\omega_n,\omega_n)} \frac{ \pr(T> a_n^T+ b_n^T(u+y))} {\pr(T > a_n^T+b_n^T u)} & \leq
  C (1+y)^{-p} \; .
\end{align}
For all $p>0$ and $z\in\Rset$, there exists a constant $C$ such that, for large enough $n$ and all
$y \geq 0$,
\begin{align}
  \label{eq:borne-uniforme-T1+T2}
  \sup_{u\in(-\omega_n,\omega_n)} K_n(u+y+z) & \leq C (1+y)^{-p} \; .
\end{align}
\end{lemma}

\begin{proof}
  Recall the representation~(\ref{eq:vonmises}). The function $\vartheta$ is upper and lower
  bounded, so
\begin{align*}
  \frac{\pr(T>a_n^T+b_n^T(u+y))}{\pr(T>a_n^T+b_n^Tu)} &
  =\frac{\vartheta(a_n^T+b_n^T(u+y))}{\vartheta(a_n^T+b_n^Tu)} \exp \left(- \int_0^y
    \frac{\psi(a_n^T+b_n^Tu)}{\psi(a_n^T+b_n^Tu+b_n^Ts)} \; \rmd s \right)  \\
  & \leq C \exp \left(- \int_0^y \frac1{1 + \psi'(\zeta_n) \frac{b_n^T} {\psi(a_n^T+b_n^Tu)} \; s}
    \; \rmd s \right) \: ,
\end{align*}
where $\zeta_n\in(a_n^T+b_n^Tu,a_n^T+b_n^Tu+b_n^Ts)$. Since $\lim_{s\to\infty}\psi'(s)=0$, and by
Assumption~\ref{hypo:strengthening-psi} $\psi(a_n^T+b_n^Tu)/b_n^T$ converges uniformly to 1 with
respect to $u\in(-\omega_n,\omega_n)$, so, for $\epsilon>0$, and large enough $n$, it holds that
\begin{align*}
  \exp \left(- \int_0^y \frac1{1 + \psi'(\zeta_n) \frac{b_n^T}{\psi(a_n^T+b_n^Tu)} \; s} \; \rmd s
  \right) \leq \exp \left(-\int_0^y \frac1{1 + \epsilon \, s} \; \rmd s \right) = (1+\epsilon
  y)^{-1/\epsilon} \; .
\end{align*}

This proves~(\ref{eq:borne-uniforme-spherique}). To prove~(\ref{eq:borne-uniforme-T1+T2}), define
$H_n(u) = n\rme^{-\omega_n} \pr(T \leq a_n^T - b_n^T \omega_n + b_n^T u)$. Then, for any fixed
$z\in\Rset$ and $y\geq0$,
\begin{align*}
  K_n(u+y+z) & = \frac{n}{\omega_n} \int_{-s_n}^{\omega_n+s_n} \pr(T>a_n^T+ b_n^T(u+y+z)) H_n(\rmd u) \\
  & \leq \frac{n}{\omega_n} \int_{-s_n}^{\omega_n+s_n} \frac{\pr(T>a_n^T+ b_n^T(u+y+z))}
  {\pr(T>a_n^T+ b_n^T(u+z))} \pr(T>a_n^T+ b_n^T(u+z)) H_n(\rmd u)  \\
  & \leq \sup_{|u|\leq \omega_n} \frac{\pr(T>a_n^T+ b_n^T(u+y+z))} {\pr(T>a_n^T+ b_n^T(u+z))} \
  K_n(z) \; .
\end{align*}
Since $K_n(z)$ is a convergent sequence for each $z\in\Rset$, it is bounded with respect to
$n$. This yields~(\ref{eq:borne-uniforme-T1+T2}).
\end{proof}

Define $c_n^T = \sqrt{b_n^T/a_n^T}$ and $d_{n}^T = \frac12({2d-k-1}) \; \log({a_n^T}/{b_n^T}) -
\log\log({a_n^T}/{b_n^T})$.
\begin{lemma}
  \label{lem:convergence-jointe-cosinus}
  For all $z\in\Rset$, $s\geq0$, $u_{i,q}\in\Rset$, $i=1,2$, $q=k+1,\dots,d$,
  \begin{multline}
    \lim_{n\to\infty} n^2 \; \pr \Big( \tfrac{\|\mby_1 -
      \mby_2\|-2\sqrt{\lambda_1}a_n^T}{\sqrt{\lambda_1}b_n^T} + d_{n}^T > z \; ,
    \tfrac{1+\langle\unifk_{1},\unifk_{2}\rangle}{(c_n^T)^2} \leq  s \; ,   \\
    \tfrac{W_{i,q}}{c_n^T} \leq  u_{i,q} \; , i=1,2,q=k+1,\dots,d \; ; \ \mct_n \Big)  \\
    = C'_k \; \rme^{-z} \, \pr(R_{k-1} \leq s) \prod_{q=k+1}^d \pr(U_{1,q} \leq u_{1,q} \, ; U_{2,q}
    \leq u_{2,q}) \; ,   \label{eq:conv-jointe-cosinus}
  \end{multline}
  where $R_{k-1}$ has a $\chi^2$ distribution with $k-1$ degrees of freedom, $(U_{1,q},U_{2,q})$ are
  independent Gaussian random vectors with marginal variance $\varpi_q^2$ and correlation $\rho_q$
  defined in~(\ref{eq:var-cor-uv}), independent of $R_{k-1}$ and
  \begin{align}
    \label{eq:def-C'k}
    C'_k = \frac{2^{d-3} (2d-k-1)\Gamma^2(\frac d2) } {\Gamma(\frac k2) \sqrt{\pi}}
    \left(\prod_{q=k+1}^d \frac{\lambda_1}{\lambda_1-\lambda_q} \right)^{1/2} \; .
  \end{align}
\end{lemma}
As a consequence, we have
\begin{align}
  \label{eq:convergence-ecart}
  \lim_{n\to\infty} & \; n^2 \; \pr \left( \frac{\|\mby_1 - \mby_2\|
      -2\sqrt{\lambda_1}a_n^T}{\sqrt{\lambda_1}b_n^T} + d_{n}^T - \log C'_k > z \; ; \mct_n \right)
  = \rme^{-z} \; .
\end{align}

\begin{proof}
  Define $C_k = \langle \unifk_1,\unifk_2\rangle$.  Let $f$ be a continuous function with
  compact support in $[0,\infty)\times\Rset^{2(d-k)}$.  The first step is to obtain a limit for
  \begin{multline*}
    E(f,z) = n^2  \esp \Big[ f \left(\tfrac{1+C_k}{(c_n^T)^2}, \tfrac{(W_{1,k+1},\dots,W_{1,d})}{c_n^T},
        \tfrac{(W_{2,k+1},\dots,W_{2,d})}{c_n^T}\right) \; ; 
    \|\mby_1-\mby_2\| > 2a_n^T-b_n^Td_n^T+b_n^Tz \; ; \mct_n \Big] \; .
  \end{multline*}
  Since $\unifk_1$ and $\unifk_2$ are independent and uniformly distributed on $\sphere{k-1}$, the
  density of the distribution of $\langle\unifk_{1},\unifk_{2}\rangle$ is $\beta_k^{-1}
  (1-s^2)^{(k-3)/2}$ on $[-1,1]$ with
  \begin{align*}
    \beta_k = \frac{\Gamma((k-1)/2)\Gamma(1/2)}{\Gamma(k/2)} \; .
  \end{align*}
  Let $g$ be the density of $(W_{k+1},\dots,W_d)$ and define
  \begin{align*}
    \tilde{K}_n(y) = n^2(c_n^T)^{2d-k-1} \pr(T_1+T_2>2a_n^T-b_n^Td_{n}^T+b_n^Ty\; ; \mct_n) \; .
  \end{align*}
  By Lemmas~\ref{lem:convergence-hn} and~\ref{lem:convergence-JJ}, $\lim_{n\to\infty} \tilde{K}_n(y)
  = \frac12(2d-k-1)\rme^{-y}$, locally uniformly with respect to $y\in\Rset$.  This yields
  \begin{align*}
    E(f,z) & = (c_n^T)^{-2d+k+1} \frac1{\beta_k}\int_{-1}^1 \int_{\Rset^{2(d-k)}} f \left( \tfrac{1 +
        \sqrt{1-\|u\|^2} \sqrt{1-\|v\|^2}s} {(c_n^T)^2}, \tfrac{u}{c_n^T}, \tfrac{v}{c_n^T} \right)    \\
    & \hspace*{2cm} \times \tilde{K}_n(z+h_n(1+s,u,v)) (1-s^2)^{(k-3)/2} g(u) g(v) \, \rmd s \, \rmd
    u \, \rmd v \\
    & = \frac1{\beta_k}\int_0^\infty \int_{\Rset^{2(d-k)}} f \left( \tfrac{1 +
        \sqrt{1-(c_n^T)^2\|u\|^2} \sqrt{1-(c_n^T)^2\|v\|^2} (-1+(c_n^T)^2t)} {(c_n^T)^2}, u,v \right)    \\
    & \hspace*{2cm} \times \tilde{K}_n(z+h_n((c_n^T)^2t, c_n^Tu, c_n^Tv) \; (2t-(c_n^T)^2t^2)^{(k-3)/2} g(c_n^Tu)
    g(c_n^Tv) \, \rmd t \, \rmd u \, \rmd v    \\
    & \to \frac{2^{(k-3)/2} g^2(0) (2d-k-1)} {2\beta_k} \int_0^\infty \int_{\Rset^{2(d-k)}} f(t,u,v)
    \rme^{-z-g(t,u,v)} \; t^{(k-3)/2} \, \rmd t \, \rmd u \, \rmd v    \\
    & = \frac{2^{(k-3)/2} g^2(0) (2d-k-1)} {2\beta_k} C''_k \rme^{-z}
    f(R_{k-1},U_{1,k+1},\dots,U_{1,d},U_{2,k+1},\dots,U_{2,d}) \; ,
  \end{align*}
  where $R_{k-1}$ has a $\chi^2$ distribution with $k-1$ degrees of freedom and is independent of the jointly Gaussian random
  variables $U_{i,q}$, $i=1,2$, $q=k+1=,\dots,d$ which are as defined in the lemma, and
\begin{align*}
  C_k'' & = 2^{(k-1)/2} \Gamma\left(\tfrac{k-1}2\right) (2\pi)^{d-k} \left(\prod_{q=k+1}^d
    \frac{\lambda_1}{\lambda_1-\lambda_q} \right)^{1/2} \; .
\end{align*}
Provided we extend this convergence to bounded continuous functions, this
yields~(\ref{eq:conv-jointe-cosinus}) with $C'_k$ as in~(\ref{eq:def-C'k}).
By Lemmas~\ref{lem:convergence-hn} and~\ref{lem:borne-uniforme-spherique}, we have, for $p>0$ and
$z\in\Rset$, there exists a constant $C$ such that 
\begin{align*}
  \tilde{K}_n(z + h_n((c_n^T)^2 t, c_n^Tw_1, c_n^Tw_2) ) & \leq C 
  \left( 1 + t + \sum_{q=k+1}^d
    (w_{1,q}^2 + w_{2,q}^2) \right)^{-p} \; 
\end{align*}
is integrable (for $p$ large) with respect to Lebesgue's measure on $[0,\infty)
\times\Rset^{2(d-k)}$. Therefore, arguing as in the proof of Theorem~\ref{theo:norm-ddim} shows that
the convergence holds for all bounded continuous functions $f$. This
proves~(\ref{eq:conv-jointe-cosinus}) by the Portmanteau theorem.
\end{proof}

We are now in a position to prove Lemma~\ref{lem:poisson-convergence}.

\begin{proof}[Proof of~(\ref{eq:poisson-convergence-esp})] 
  Let $C_k$ be as in~(\ref{eq:def-Ck}). Then $\log C_k = \log C'_k - 2\log D_k -
  \log2$. Plug these values and the expression of $a_n$ in terms of $a_n^T$ and
  $b_n^T$ obtained in~(\ref{eq:quantile-relation})
  into~(\ref{eq:convergence-ecart}) and note that $\log(a_n/b_n) =
  \log(a_n^T/b_n^T) + o(1)$.
\end{proof}

\begin{proof}[Proof of~(\ref{eq:poisson-convergence-neglige})]
  For $y_i = t_i (\sqrt{\lambda_1}
  w_{i,1},\dots,\sqrt{\lambda_d}w_{i,d})$, $i=1,2$ and $z\in\Rset$,
  define
  \begin{align*}
    f_n(y_1,y_2) =\1{\{\|y_1-y_2\|>2a_n-b_nd_n+b_nz\}}\1{\{t_1\vee t_2\leq a_n^T + b_n^T s_n\}} \; .
  \end{align*}
  Then
  \begin{align}
    \label{eq:on3-business}
    \esp[f_n(\mby_1,\mby_2) f_n(\mby_1,\mby_3)] = o(n^{-3}) \; .
  \end{align}
  Using the notation of Lemma~\ref{lem:convergence-jointe-cosinus}, we have, for some constant
  $C>0$, 
  \begin{multline*}  
    \esp[f_n(\mby_1,\mby_2) f_n(\mby_1,\mby_3)]
    = \esp \Big[ \Big( \esp[f_n(\mby_1,\mby_2) \mid \mby_1] \Big)^2 \Big] \\
    \leq C  n^{-4} (c_n^T)^{k-d} \int_{\Rset^{d-k}} \Big( \int_0^\infty \Big\{ \int_{\Rset^{d-k}}
    \tilde{K}_n(z + h_n\{(c_n^T)^2 t, c_n^T u, c_n^T    v\} ) \\
    \times t^{(k-3)/2} g(c_n^Tu) \rmd u \Big\} \rmd t \Big)^2 \times g(c_n^Tv) \,
    \rmd v \; .
  \end{multline*}

  By the same arguments as in the proof of Lemma~\ref{lem:convergence-jointe-cosinus}, the integral
  converges to a constant times
\begin{align*}
  \int_{\Rset^{d-k}} \left( \int_0^\infty \int_{\Rset^{d-k}} \rme^{-z-g(t,u,v)} t^{(k-3)/2} \rmd t
    \rmd u_{k+1}\dots\rmd u_d \right)^2 \rmd v_{k+1}\dots\rmd v_d\; .
\end{align*}

This yields that $n^3 \esp[f_n(\mby_1,\mby_2) f_n(\mby_1,\mby_3)]=O((c_n^T)^{k-d}/n) = o(1)$ since
$a_n^T/b_n^T$ is always a slowly varying sequence which implies that $(a_n^T/b_n^T)^p = o(n)$ for
any $p>0$. Indeed, define $\chi(x) = \{\frac{F^\leftarrow}{ \psi\circ F^\leftarrow}\} (1-1/x)$. We
can assume that $\psi$ is differentiable with $\psi'(x)\to0$. Then it suffices to prove that
$\lim_{x\to\infty}x\chi'(x)/\chi(x)=0$. An elementary computation yields, with
$y=F^\leftarrow(1-1/x)$,
\begin{align*}
  \frac{x\chi'(x)}{\chi(x)} = \frac{\psi(y)}{y} -\psi'(y) \to 0 \; ,
\end{align*}
as $x$, hence $y$, tends to infinity. Thus $\chi$ is slowly varying at infinity.
\end{proof}

\paragraph{Acknowledgement} The simulations were made with the \Rlogo\ package diameter written by
Bernard Desgraupes available at \url{http://bdesgraupes.pagesperso-orange.fr/R.html}. Bernard
Desgraupes' help is gratefully acknowledged. We also thank Enkelejd Hashorva for bringing the
reference \cite{hashorva:korshunov:piterbarg:2013} to our attention and for pointing out a mistake
in the first version.

\end{document}